\documentclass[11pt]{amsart}
\usepackage{amsfonts,amsmath,amssymb}
\usepackage{mathabx}
\usepackage{mathrsfs}
\usepackage[latin1]{inputenc}
\usepackage[usenames,dvipsnames]{pstricks}
\newtheorem{theorem}{Theorem}[section]

\newtheorem{Corol}[theorem]{Corollary}

\newtheorem{Lemm}[theorem]{Lemma}

\newtheorem{proposition}[theorem]{Proposition}

\newcommand{\p}{\partial}
\newcommand{\R}{\mathbb{R}}

\newcommand{\para}[1]{\left(#1\right)}

\usepackage{graphicx}

\usepackage{mathrsfs}
\usepackage[latin1]{inputenc}
\usepackage[T1]{fontenc}
\usepackage{color,xspace}
\usepackage[usenames,dvipsnames]{pstricks}
\usepackage{times}
\allowdisplaybreaks
\begin{document}

\title[]{Optimal stability for a first order coefficient in a non-self-adjoint wave equation from Dirichlet-To-Neumann map}
\author[M.~Bellassoued]{Mourad~Bellassoued}
\author[I.~Ben A\"{\i}cha]{Ibtissem Ben A\"{\i}cha}
\address{M.~Bellassoued.  University of Tunis El Manar, National
Engineering School of Tunis, ENIT-LAMSIN, B.P. 37, 1002 Tunis,
Tunisia}
\email{mourad.bellassoued@enit.utm.tn }
\address{I.~Ben A\"{\i}cha.University of Tunis El Manar, National
Engineering School of Tunis, ENIT-LAMSIN, B.P. 37, 1002 Tunis,
Tunisia}
\email{ibtissem.benaicha@enit.utm.tn}
\maketitle 


\begin{abstract}
This paper is focused on the study of an inverse problem for a non-self-adjoint hyperbolic equation. More precisely, we attempt to stably recover a first order coefficient appearing in a wave equation from the knowledge of Neumann boundary data. We show in dimension $n$ greater than two, a stability estimate of H\"older type for the inverse problem under consideration. The proof involves the reduction to an auxiliary inverse problem for an electro-magnetic wave equation and the use of an appropriate Carleman estimate.
\end{abstract}


\section{Introduction and  main results}

The main purpose of this paper is the study of an inverse problem of determining  a coefficient of order one on space appearing  in a non-self-adjoint wave equation. Let $\Omega\subset \R^{n}$ with  $n\geq2$, be an open  bounded set with a sufficiently smooth boundary $\Gamma=\p\Omega$. For $T>0$, we denote  by $Q=\Omega\times(0,T)$ and $\Sigma=\Gamma\times(0,T)$. We introduce the following initial boundary value problem  for the wave equation with a velocity field $V$,
\begin{equation}\label{equation n 1.1}
\left\{
  \begin{array}{ll}
   \mathscr{L}_{V}u:=(\p_{t}^{2}-\Delta+V\cdot \nabla)u=0 & \mbox{in} \,Q, \\
   \\
    u_{|t=0}=\p_{t}u_{|t=0}=0 & \mbox{in}\, \Omega,  \\
    \\
    u=f & \mbox{on} \,\Sigma,
  \end{array}
\right.
\end{equation}  
where  $V\!\in\! W^{1,\infty}(\Omega,\R^{n})$ is a real vector field and 
 $f\!\in\! \mathcal{H}^{1}_{0}(\Sigma)\!:=\!\{ f\!\in\! H^{1}(\Sigma),\,\,f_{|t=0}=0\}$ is the Dirichlet data that is used to probe the system. We may define the  so-called Dirichlet-to-Neumann (DN) map  associated with the wave operator $\mathscr{L}_{V}$ as follows 
$$\begin{array}{ccc}
\Lambda_{V}: \mathcal{H}_{0}^{1}(\Sigma)&\longrightarrow& L^{2}(\Sigma)\\
f&\longmapsto &\p_{\nu}u,
\end{array}$$
  where $\nu$ denotes the unit outward normal to $\Gamma$ at $x$ and $\p_{\nu}u$ stands for $\nabla u\cdot \nu$. 

The inverse problem we address is to determine the velocity field $V$ appearing in  (\ref{equation n 1.1}) from the knowlegde of the  DN map $\Lambda_{V}$ and we aim to derive a stability result for this problem. To our knowledge this paper is the first treating the recovery of a coefficient of order one on space  appearing in a wave equation.

The problem of recovering coefficients appearing in hyperbolic equations gained increasing popularity among mathematicians within the last few decades and there are many works related to this topic. But they are mostly concerned with coefficients of order zero on space. In the case where the unknown coefficient is depending only on the spatial variable, Rakesh and Symes \cite{[BIB19]} proved
by means of geometric optics solutions, a uniqueness result in recovering a
time-independent potential in a wave equation from global Neumann data. The uniqueness by local Neumann data, was considered by Eskin \cite{[BIB13]} and
Isakov \cite{[BIB15]}.  In \cite{[BIB4]}, Bellassoued, Choulli and Yamamoto
proved a log-type stability estimate, in the case where the Neumann data are
observed on any arbitrary subset of the boundary.  Isakov and Sun
\cite{[BIB16]} proved that the knowledge of local Dirichlet-to-Neumann map
yields a stability result of H\"older type in determining a coefficient in a
subdomain. As for stability results obtained from global Neumann data, one can
see Sun \cite{[BIB22]}, Cipolatti and Lopez \cite{[BIB12]}. There are also growing publications on the related inverse problems in Riemannian case. We mention e.g the paper of Bellassoued and Dos  Santos Ferreira \cite{[BIB5]}, Stefanov and Uhlmann \cite{[BIB21]} and \cite{[BA4]} in which Liu and Oksanen consider the problem of recovering  a wave speed $c$ from acoustic boundary measurements modelled by the hyperbolic Dirichlet to Neumann map.  Other than the mentioned papers, the recovery of time-dependent coefficients in hyperbolic equations has also been developped recently, we refer e.g to Bellassoued and Ben A\"icha \cite{[BIB1],[BIB2]} and in the Riemmanian case, we refer to the work of Waters \cite{[BA1]}, in which a stability of  H\"older type was proved, for the identification of  the $X$-ray transform of a time-dependent coefficient in an hyperbolic equation.  In \cite{[BA2]}, R. Salazar considered  the stability issue and  extended the result of the paper  \cite{[BA3]} to more general coefficients and he established a stability result for compactly supported coefficients provided $T$ is sufficiently large. For curiosity, the reader can also see \cite{[BIB8],[BIB17]} and the references therein.

The above papers are concerned only with coefficients of order zero on space. In the case where the unknown coefficient is of order one, we cite for example the paper of  Pohjola \cite{[BIB18]}, in which he considered an inverse problem for a steady state convection diffusion equation. He showed by reducing his problem to the case of a stationary magnetic Schr\"odinger equation that a velocity field can be uniquely determined from the knowledge of Neumann measurements. Cheng, Nakamura and Somersalo \cite{[BIB11]} treated the same problem and they proved a uniqueness result for more regular coefficients. Salo \cite{[BIB20]} also studied this  problem and proved a uniquness result in the case where the coefficient is Lipschitz continuous.  The overall method of proving uniqueness in these papers  was based on reducing the inverse problems under investigation to similar ones for self-adjoint operators and applying the maximum principle. We can also refer to the paper \cite{[BA5]} in which a uniqueness result for a general non-self-adjoint second-order elliptic operator on a manifold with boundary is addressed. 

  The stability  for problems associated with non self-adjoint  operators is never treated before.  In this  work, we consider this challenging problem and we establish a stability  estimate of  H\"older type for the recovery of  the first order coefficient $V$ appearing in the  wave operator $\mathscr{L}_{V}$ from the knowledge of the DN map  $\Lambda_{V}$.  The proof of the stability estimate requires the use of  an $L^{2}$-weighted inequality called a Carleman estimate designed for elliptic operators (see \cite{[BIB7],[BIB10]} ) instead of the maximum principle used in \cite{[BIB18]}. 
  
  Before stating our main result, we introduce the admissible set of the coefficients $V$. Given $M\!>\!0$ and $V_{0}\!\in \!W^{1,\infty}(\Omega, \R^{n})$, we define  
$$\mathcal{V}(M,V_{0}):=\{V\in W^{1,\infty}(\Omega,\R^{n}),\,\,\|V\|_{W^{1,\infty}(\Omega)}\leq M,\,\, \color{black} V=V_{0}\,\,\mbox{on}\,\,\Gamma\color{black} \}.$$
Then our main result can be stated as follows
\begin{theorem}\label{Theorem1.1}
Let $V_{1},\,V_{2}\in \mathcal{V}$ such that $V_{1}-V_{2}\in W^{2,\infty}(\Omega)$. Then, there exist positive constants $\kappa\in(0,1)$ and $C>0$ such that 
$$\|V_{1}-V_{2}\|_{L^{2}(\Omega)}\leq C \|\Lambda_{V_{1}}-\Lambda_{V_{2}}\|^{\kappa}.$$
Here the constant $C$ is depending only on $\Omega$ and $M$ and $\|\cdot\|$ denotes the norm in $\mathcal{L}(\mathcal{H}^{1}_{0}(\Sigma);L^{2}(\Sigma))$.
\end{theorem}
The above statement claims stable determination of the velocity field  $V$ from the knowledge of the DN map $\Lambda_{V}$, where both the Dirichlet and Neumann data are performed on the whole boundary $\Sigma.$  By Theorem \ref{Theorem1.1}, we can readily derive the following  
\begin{Corol} 
Let $V_{1},\,V_{2}\in \mathcal{V}$. Then, we have that $\Lambda_{V_{1}}=\Lambda_{V_{2}}$ implies $V_{1}=V_{2}$ everywhere in $\Omega$.
\end{Corol} 
We point out that since the hyperbolic operator $\mathscr{L}_{V}$ is not self-adjoint, then we should first head toward an auxiliary problem  for an electro-magnetic wave equation in order to be able to prove our main results.  

The remainder of this paper is organized as follows: in Section \ref{Sec2}, we reduce the inverse problem associated with the equation (\ref{equation n 1.1}) to a corresponding inverse  problem for an electro-magnetic wave equation. By the use of an elliptic Carleman estimate, we give in Section \ref{Sec3} the proof of Theorem \ref{Theorem1.1}. 


\section{Reduction of the problem}\label{Sec2}
The overall method of proving the stability  for the inverse problem under consideration is mainly  based on reducing it to an equivalent problem concerning  the following electro-magnetic wave equation
\begin{equation}\label{equation n 2.2}
\left\{
  \begin{array}{ll}
   \mathscr{H}_{A,q}u:=(\p_{t}^{2}-\Delta_{A}+q)u=0 & \mbox{in} \,\,\,Q, \\
   \\
    u_{|t=0}=\p_{t}u_{|t=0}=0 & \mbox{in}\,\,\, \Omega,  \\
   \\ 
    u=f & \mbox{on} \,\Sigma,
  \end{array}
\right.
\end{equation}
where $f\in \mathcal{H}^{1}_{0}(\Sigma)$ is a non homogeneous Dirichlet data,  $A=(a_{j})_{1\leq j\leq n}\in W^{1,\infty}(\Omega,\mathbb{C}^{n})$ is  a  pure imaginary  complex magnetic vector and $q\in L^{\infty}(\Omega,\mathbb{R})$ is a bounded electric potential.
Here $\Delta_{A}$  denotes the magnetic Laplacien and it is given by 
 $$\Delta_{A}=\sum_{j=1}^{n}(\p_{j}+ia_{j})^{2}=\Delta +2i A\cdot \nabla+i\mbox{div}\,A-A\cdot A.$$
According to \cite{[BIB3], [BIB7], [BIB9]}, the initial boudary value problem (\ref{equation n 2.2}) is  well posed and we have the existence of a unique solution within the following class
$u\in \mathcal{C}([0,T];H^{1}(\Omega))\cap \mathcal{C}^{1}([0,T];L^{2}(\Omega)).$    
 Therefore, we may define the DN map $N_{A,q}$  associated with the wave equation (\ref{equation n 2.2}) as follows 
 $$\begin{array}{ccc}
N_{A,q}:\mathcal{H}^{1}_{0}(\Sigma) &\longrightarrow& L^{2}(\Sigma)\\
f&\longmapsto& (\p_{\nu}+iA\cdot \nu)u.
\end{array}$$ 
The purpose of this section is to reduce the inverse problem associated with the wave equation (\ref{equation n 1.1}) to an auxiliary problem for (\ref{equation n 2.2}). Note that if $A$ is of real valued then $\mathscr{H}_{A,q}$ is a self-adjoint wave operator. The strategy is mainly inspired by \cite{[BIB18],[BIB11],[BIB20]}. We specify the choice of the pure imaginary complex vector $A$ and the real function $q$  in such a way $\mathscr{H}_{A,q}$ co\"incide with $\mathscr{L}_{V}$  and the same for the associated DN maps. 

We need first to introduce some notations. Let us  consider the following set  
$$ \mathcal{H}^{1}_{T}(\Sigma)\!:=\!\{g\!\in \!H^{1}(\Sigma);\,\,g_{|t=T}=0\}.$$
 For $g\in H^{1}_{T}(\Sigma)$, we define the adjoint operator of $\Lambda_{V}$ as follows:
$$\begin{array}{ccc}
\Lambda^{*}_{V}: \mathcal{H}_{T}^{1}(\Sigma)&\longrightarrow& L^{2}(\Sigma)\\
g&\longmapsto &\p_{\nu}v,
\end{array}$$
where $v$ here denotes the unique solution of the backward problem
\begin{equation*}
\left\{
  \begin{array}{ll}
   \mathscr{L}^{*}_{V}v=\p_{t}^{2}v-\Delta v-\mbox{div}(Vv)=0 & \mbox{in} \,\,\,Q, \\
  \\
    v_{|t=T}=\p_{t}v_{|t=T}=0 & \mbox{in}\,\,\, \Omega,  \\
   \\
    v=g& \mbox{on} \,\,\,\Sigma.
  \end{array}
\right.
\end{equation*}
On the other hand, we define the adjoint operator of the DN map $N_{A,q}$ as follows 
$$\begin{array}{ccc}
N^{*}_{A,q}: \mathcal{H}_{T}^{1}(\Sigma)&\longrightarrow& L^{2}(\Sigma)\\
g&\longmapsto &(\p_{\nu}+iA\cdot \nu)v,
\end{array}$$
associated to the backward problem 
\begin{equation*}
\left\{
  \begin{array}{ll}
  \mathscr{H}^{*}_{A,q}v=\mathscr{H}_{\overline{A},q}v=0 & \mbox{in} \,\,\,Q, \\
  \\
    v_{|t=T}=\p_{t}v_{|t=T}=0 & \mbox{in}\, \Omega,  \\
   \\
    v=g& \mbox{on} \,\Sigma.
  \end{array}
\right.
\end{equation*} 
In the sequel, we shall make use of the following Green formula for the magnetic Laplacian. Let $A$ be a pure imaginary complex vector in $W^{1,\infty}(\Omega,\mathbb{C}^{n})$. Then, the following identity holds true
 \begin{eqnarray}\label{equation n 2.3}
 \int_{\Omega}\Delta_{A}u\,\overline{v}\,dx&=&-\int_{\Omega}(\nabla+i A)u\, \overline{(\nabla -i A)v}\,dx +\int_{\Gamma}(iA\cdot \nu+\p_{\nu})u \overline{v}\,d\sigma\cr
 &=& \int_{\Omega}\overline{\Delta_{\overline{A}}v}\,u\,dx+\int_{\Gamma}\Big( (\p_{\nu}+i\nu\cdot A)u\,\overline{v}-\overline{(\p_{\nu}+i\nu\cdot\overline{A})v}\,u\,\Big )d\sigma,
 \end{eqnarray}
 for $u,\,v\in H^{1}(\Omega)$ such that $\Delta u,\,\Delta\,v\in L^{2}(\Omega)$. Here $d\sigma$ is the Euclidean surface measure on $\Gamma$. 
Finally, we introduce the admissible sets of the coefficients $A$ and $q$: for $M>0$, $A_{0}\in {W}^{1,\infty}(\Omega,\mathbb{C}^{n})$ and  $q_{0}\in L^{\infty}(\Omega,\R)$, we define
$$\mathcal{A}(M,A_{0}):=\{A\in W^{1,\infty}(\Omega,\mathbb{C}^{n}),\,\, \|A\|_{W^{1,\infty}(\Omega)}\leq M,\,\, A=A_{0}\,\,\mbox{on}\,\,\Gamma \},$$
and   
$$\mathcal{Q}(M,q_{0}):=\{q\in L^{\infty}(\Omega,\R),\,\, \|q\|_{L^{\infty}(\Omega)}\leq M,\,\, \color{black} q=q_{0}\,\,\mbox{on}\,\, \Gamma \color{black} \}.$$ 
We shall now  give  some properties of the considered operators as well as the associated DN maps. This statement  will play a crucial role in proving Theorem \ref{Theorem1.1}.
\begin{Lemm}\label{Lemma 2.1}
 Let $V_{1},\,V_{2}\in\mathcal{V}$. We define  $A_{1},\,A_{2}\in\mathcal{A}(M,A_{0})$ and $q_{1},\,q_{2}\in\mathcal{Q}(M,q_{0})$ by 
 \begin{equation}\label{equation  n 2.4}
  A_{j}=\frac{i}{2}V_{j},\quad 
 \mbox{and} \quad q_{j}=\frac{1}{4}{V_{j}^{2}}-\frac{1}{2}\,\mbox{div}\, V_{j},\quad j=1,\,2.
 \end{equation}
 Then, we have 
 \begin{equation*}
 \mathscr{H}_{A_{j},q_{j}}=\mathscr{L}_{V_{j}},\quad \quad \mathscr{H}^{*}_{A_{j},q_{j}}=\mathscr{L}^{*}_{V_{j}},\quad \mbox{and} \quad N_{\overline{A}_{j},q_{j}}=N_{-A_{j},q_{j}}=N_{A_{j},q_{j}}^{*},\quad j=1,\,2 .
 \end{equation*}  
Moreover, 
\begin{equation}\label{equation  n 2.5}
\|N_{A_{1},q_{1}}-N_{A_{2},q_{2}}\|=\|\Lambda_{V_{1}}-\Lambda_{V_{2}}\|,
\end{equation}
where $\|\cdot\|$ stands for the norm in $\mathcal{L}(\mathcal{H}^{1}_{0}(\Sigma);L^{2}(\Sigma))$.
 \end{Lemm}
 \begin{proof}
 In light of (\ref{equation  n 2.4}), one can easily see that for any  $u,\,v\in H^{2}(Q)$ we have 
 \begin{equation}\label{equation  n 2.6}
 \mathscr{H}_{A_{j},q_{j}}u=(\p_{t}^{2}-\Delta _{A_{j}}+q_{j}(x))u=(\p_{t}^{2}-\Delta+V_{j}\cdot \nabla)u=\mathscr{L}_{V_{j}}u,
 \end{equation}
 and 
 \begin{equation}\label{equation  n 2.7}
 \mathscr{H}^{*}_{A_{j},q_{j}}v=(\p_{t}^{2}-\Delta_{\overline{A}_{j}}+q_{j}(x))v=(\p_{t}^{2}-\Delta_{(-A_{j})}+q_{j}(x))v=\p_{t}^{2}v-\Delta v -\mbox{div}(V_{j} v)= \mathscr{L}^{*}_{V_{j}}v.
 \end{equation} 
 A simple application of (\ref{equation n 2.3}) yields 
$N_{\overline{A},q}=N_{-A,q}=N_{A,q}^{*}$. 
 We move now to prove (\ref{equation  n 2.5}). Let us  denote by $u_{j}$ and $v_{j}$, $j=1,\,2$,  the solutions of
\begin{equation}\label{equation  n 2.8}
\left\{
  \begin{array}{ll}
  \mathscr{L}_{V_{j}}u_{j}=0 & \mbox{in} \,\,\,Q, \\
  \\
    u_{j|t=0}=\p_{t}u_{j|t=0}=0 & \mbox{in}\,\,\, \Omega,  \\
   \\
    u_{j}=f & \mbox{on} \,\,\,\Sigma,
  \end{array}
\right. ;\qquad \qquad \left\{
  \begin{array}{ll}
  \mathscr{L}^{*}_{V_{j}}v_{j}=0 & \mbox{in} \,\,\,Q, \\
  \\
    v_{j|t=T}=\p_{t}v_{j|t=T}=0 & \mbox{in}\,\,\, \Omega,  \\
   \\
    v_{j}=g & \mbox{on} \,\,\,\Sigma,
  \end{array}
\right.
\end{equation} 
where $f\in\mathcal{H}^{1}_{0}(\Sigma)$ and $g\in \mathcal{H}^{1}_{T}(\Sigma)$. 
By multiplying the first equation in the left hand side of (\ref{equation  n 2.8}) by $\overline{v}_{j}$ and integrating by parts, we get
\begin{equation}\label{equation  n 2.9}
 \int_{\Sigma}\Lambda_{V_{j}}(f) \overline{g}\,d\sigma\,dt=\int_{Q}\Big( -\p_{t}u_{j}\p_{t}\overline{v}_{j}+\nabla u_{j}\cdot \nabla \overline{v}_{j}+V_{j}\cdot \nabla u_{j}\,\overline{v}_{j}\Big) \,dx\,dt.
 \end{equation}
 On the other hand, based on (\ref{equation  n 2.6}) and (\ref{equation  n 2.7}), $u_{j}$  and $v_{j}$ with  $j=1,\,2$, are also solutions to 
 \begin{equation}\label{equation  n 2.10}
\left\{
  \begin{array}{ll}
  \mathscr{H}_{A_{j},q_{j}}u_{j}=0 & \mbox{in} \,\,\,Q, \\
  \\
    u_{j|t=0}=\p_{t}u_{j|t=0}=0 & \mbox{in}\,\,\, \Omega,  \\
   \\
    u_{j}=f & \mbox{on} \,\,\,\Sigma,
  \end{array}
\right.; \qquad\qquad \left\{
  \begin{array}{ll}
   \mathscr{H}^{*}_{A_{j},q_{j}}v_{j} =0 & \mbox{in} \,\,\,Q, \\
  \\
    v_{j|t=T}=\p_{t}v_{j|t=T}=0 & \mbox{in}\,\,\, \Omega,  \\
   \\
    v_{j}=g & \mbox{on} \,\,\,\Sigma.
  \end{array}
\right.
\end{equation}
By multiplying the equation in the left hand side of (\ref{equation  n 2.10}) by $\overline{v}_{j}$ and after  integrating  by parts, we get in light of (\ref{equation  n 2.4}) and (\ref{equation n 2.3}),
$$\int_{\Sigma}N_{A_{j},q_{j}}(f) \overline{g}\,d\sigma\,dt=\int_{Q}\Big( -\p_{t}u_{j}\p_{t}\overline{v}_{j}+\nabla u_{j}\cdot \nabla \overline{v}_{j}+\frac{1}{2}V_{j}\cdot \nabla u_{j} \overline{v}_{j}-\frac{1}{2}V_{j}\cdot \nabla \overline{v}_{j} u_{j}-\frac{1}{2}\mbox{div}V_{j}u_{j}\overline{v}_{j}\Big) \,dx\,dt.$$
This immediately implies that
\begin{equation}\label{equation  n 2.11}
\int_{\Sigma} N_{A_{j},q_{j}}(f) \overline{g}\,d\sigma\,dt=\int_{Q}\Big( -\p_{t}u_{j}\p_{t}\overline{v}_{j}+\nabla u_{j}\cdot \nabla \overline{v}_{j}+V_{j}\cdot \nabla u_{j}\overline{v}_{j}\Big) \,dx\,dt-\frac{1}{2}\int_{\Sigma}V_{j}\cdot \nu \,u_{j}\overline{v}_{j}\,d\sigma\,dt.
\end{equation}
Hence, from (\ref{equation  n 2.9}) and (\ref{equation  n 2.11}), we find out that
$$\int_{\Sigma} N_{A_{j},q_{j}}(f) \overline{g}\,d\sigma\,dt=\int_{\Sigma}\Lambda_{V_{j}}(f) \overline{g}\,d\sigma\,dt-\frac{1}{2}\int_{\Sigma}V_{j}\cdot \nu\,f\,g\,d\sigma\,dt.$$
Owing to the assumption that $V_{1}=V_{2}$ on $\Gamma$, we get the desired result.
 \end{proof}

Due to Lemma \ref{Lemma 2.1},  the inverse problem under investigation may be equivalently reformulated as to whether the magnetic potential $A$ and the electric potential $q$ in (\ref{equation n 2.2}) can be recovered from the knowledge of $N_{A,q}$. This is the auxiliary inverse problem that we address in the remaining of this section.

As it was noted in Sun \cite{[BIB23]}, the DN map is invariant under a gauge transformation. Namely, given any $\varphi\in \mathcal{C}^{2}(\overline{\Omega})$ with $\varphi_{|\Gamma}=0$, one has $N_{A+\nabla\varphi,q}=N_{A,q}$. Hence, the magnetic potential $A$ can not be uniquely determined by $N_{A,q}$. However it is possible to show that the knowledge of the DN map $N_{A,q}$ stably determines the electric potential $q$ and the magnetic field corresponding to the pure imaginary complex potential $A$ which  is given by the $2$-form $d\alpha_{A}$ defined as follows
$$d\alpha_{A}=\sum_{i,j=1}^{n} \Big(\frac{\p a_{i}}{\p x_{j}}-\frac{\p a_{j}}{\p x_{j}}  \Big)dx_{j}\wedge dx_{i}. $$
Actually, this problem is closely related to the one treated in Bellassoued and Ben Joud \cite{[BIB3]} in the absence of the electric potential, in  Bellassoued \cite{[BA6]}  in the Riemmanian case  
  and in Ben Joud \cite{[BIB9]}.  Compared with the paper of Ben Joud \cite{[BIB9]}, we formulate this auxiliary problem for \textit{less} regular \textit{complex} magnetic potentials.

Theorem \ref{Theorem1.1} can then be reduced to the following equivalent statement
\begin{theorem}\label{prop}
Let $A_{1},\,A_{2}\in \mathcal{A}(M,A_{0})$, and $q_{1},\,q_{2}\in \mathcal{Q}(M,q_{0})$.   Assume that $A_{1}-A_{2}\in W^{2,\infty}(\Omega,\mathbb{C}^{n})$ and  $q_{1}-q_{2}\in W^{1,\infty}(\Omega,\R).$ Then, there exist $C>0$ and $\mu\in(0,1)$ such that we have 
$$\|d\alpha_{{A_{1}}}-d\alpha_{A_{2}}\|_{H^{-1}(\Omega)}+\|q_{1}-q_{2}\|_{H^{-1}(\Omega)}\leq C\|N_{A_{2},q_{2}}-N_{A_{1},q_{1}}\|^{\mu}.$$ 
\end{theorem}
The above theorem claims stable determination of the magnetic field $d\alpha_{A}$  and the electric potential $q$ from the global Neumann measurement $N_{A,q}$.  Here we improve the result of Ben Joud \cite{[BIB9]} by considering  complex magnetic  potentials. The regularity condition imposed on admissible magnetic potentials is also weakened from $W^{3,\infty}(\Omega)$ to $W^{1,\infty}(\Omega)$.  

The rest of this section is devoted to proving this auxiliary result.
\subsection{Geometrical optics solutions}
Section \ref{Sec2} mainly aims at the study of the auxiliary inverse problem associated with the electro-magnetic wave equation (\ref{equation n 2.2}), that is the identification of $d\alpha_{A}$ and $q$ from the DN map $N_{A,q}$. To begin with, we shall first construct   geometrical optics solutions for the equation (\ref{equation n 2.2}) associated  with a suitable smooth approximation of the magnetic potential (see \cite{[BIB6],[BIB18]}). For this purpose,  we first consider $\varphi\in \mathcal{C}_{0}^{\infty}(\R^{n})$ and notice that for all  $\omega\in \mathbb{S}^{n-1}$ the function  
\begin{equation}
\label{equation  n 2.12}
\phi(x,t)=\varphi(x+t\omega),
\end{equation}
solves the following transport equation 
$$(\p_{t}-\omega\cdot \nabla)\phi(x,t)=0.$$ 
 We will build solutions associated with a suitable smooth approximation of the magnetic potential. This requires to extend the potentials  $A_{1},\,A_{2}\in \mathcal{A}(A_{0},M)$ to a larger domain as follows:
\begin{Lemm}(see\cite{[BIB24]})
Let $\Omega$ be a bounded domain that is compactly contained in $\tilde{\Omega}\subset \R^{n}$. Let  $A_{1},\,A_{2}\in W^{1,\infty}(\Omega)$ such that $\|A_{j}\|_{W^{1,\infty}(\Omega)}\leq M$, $j=1,\,2$ and $A_{1}=A_{2}$ on $\Gamma$. Then, there exist two extensions $\tilde{A}_{1},\,\tilde{A}_{2}\in W^{1,\infty}_{c}(\tilde{\Omega})$,  such that $\tilde{A}_{1}=\tilde{A}_{2}$ on $\tilde{\Omega}\,\setminus \Omega$. Moreover, there exists a positive constant $C>0$  such that 
$$\|\tilde{A}_{j}\|_{W^{1,\infty}(\tilde{\Omega})}\leq C\,M,\quad j=1,\,2.$$
Here $C$ is depends only on $\Omega$, $\tilde{\Omega}$ and $M$.
\end{Lemm}
Let $\chi\in \mathcal{C}^{\infty}_{c}(\R^{n})$ such that Supp $\chi\subset\! B(0,1)$, $0\leq \chi\leq 1$, and $\displaystyle\int_{\R^{n}}\chi(x)\,dx=1$.  For a sufficiently large $\lambda>0$, we denote $\chi_{\lambda}(x)=\lambda^{n\color{black}{\alpha}\color{black}}\chi(\lambda^{\color{black}\alpha}\color{black} x)$, with \color{black} $0<\alpha\leq 1/2 $\color{black}. For $j=1,\,2$, we define the smooth approximations $A^{\sharp}_{j,\lambda}$ of the extensions $\tilde{A}_{j}\in W^{1,\infty}_{c}(\R^{n},\mathbb{C})$ as follows:
\begin{equation}\label{equation  n 2.13}
A^{\sharp}_{j,\lambda}:= \chi_{\lambda}\ast \tilde{A}_{j},\quad j=1,\,2.
\end{equation}
This terminology is justified by the fact that  $A^{\sharp}_{j,\lambda}$ gets closer to $\tilde{A}_{j}$ as $\lambda$ goes to $\infty$. This can be seen from the following result:
\begin{Lemm}
Let $\tilde{A}\in W^{1,\infty}_{c}(\R^{n},\mathbb{C}^{n})$ be  such that $\|\tilde{A}\|_{W^{1,\infty}(\R^{n})}\leq M$. Then, there exists a positive constant $C$ depending only on $M$ and $\Omega$ such that for all $\lambda>0$ we have 
\begin{equation}\label{equation  n 2.14}
\|\tilde{A}-A^{\sharp}_{\lambda}\|_{L^{\infty}(\R^{n})}\leq C \lambda^{-\alpha}.
\end{equation}
Moreover, for any multi-index $\gamma\in\mathbb{N}^{n}$, with $|\gamma|\geq 1$, we have 
\begin{equation}\label{equation  n 2.15}
 \|\p^{\gamma}A^{\sharp}_{\lambda}\|_{L^{\infty}(\R^{n})}\leq C \lambda^{\alpha(|\gamma|-1)},
\end{equation}
where $C$ is a positive constant  depedning only on $M$ and $\Omega$.
\end{Lemm} 
\begin{proof}
From \cite{[BIB14]}, we have $\tilde{A}\in \mathcal{C}^{0,1}(\R^{n})$ and  $\|\tilde{A}\|_{\mathcal{C}^{0,1}(\R^{n})}\leq \|\tilde{A}\|_{W^{1,\infty}(\R^{n})}$, thus in light of (\ref{equation  n 2.13}), we have 
$$\begin{array}{lll}
|\tilde{A}-A^{\sharp}_{\lambda}|&=&\Big|\displaystyle\int_{\R^{n}}\tilde{A}(x-y) \chi_{\lambda}(y)\,dy-\tilde{A}(x)  \Big|\\
&=&\Big| \displaystyle\int_{\R^{n}}\tilde{A}(x-y) \chi(\lambda^{\alpha }y ) \lambda^{n\alpha}\,dy-\tilde{A}(x) \Big|\\
&=&\Big|\displaystyle\int_{\R^{n}}\tilde{A}(x-\lambda^{-\alpha}y)\chi(y)-\tilde{A}(x) \chi(y)\,\,dy   \Big|\\
&\leq&\|\tilde{A}\|_{\mathcal{C}^{0,1}(\R^{n})}\lambda^{-\alpha}\displaystyle\int_{\R^{n}}|y||\chi(y)|\,dy\\
&\leq & C \lambda^{-\alpha},
\end{array}$$
which completes the proof of the first estimate (\ref{equation  n 2.14}). We move now to prove (\ref{equation  n 2.15}) . We should first notice that for all multi-index $\gamma\in\mathbb{N}^{n}$ such that $|\gamma|\geq 1$, we have 
$$\int_{\R^{n}}\p^{\gamma}\chi(y)\,dy=0.$$
Thus, from the above observation, we find
$$\begin{array}{lll}
|\p^{\gamma}A_{\lambda}^{\sharp}(x)|&=&\Big| \displaystyle\int_{\R^{n}}\tilde{A}(y)\lambda^{n\alpha}\p^{\gamma}\Big(\chi (\lambda^{\alpha}(x-y))\Big)\,dy  \Big|\\
&=&\Big|\displaystyle\int_{\R^{n}} \tilde{A}(x-\lambda^{-\alpha}z)\lambda^{\alpha |\gamma|}\,(\p^{\gamma}\chi)(z)\,dz \Big|\\
&=&\Big| \displaystyle\int_{\R^{n}} \Big( \tilde{A}(x-\lambda^{-\alpha}z)-A(x)\Big)\,\lambda^{\alpha |\gamma|}\,(\p^{\gamma}\chi)(z)\,dz  \Big|\\
&\leq& \|\tilde{A}\|_{\mathcal{C}^{0,1}(\R^{n})}\lambda^{\alpha(|\gamma|-1)}\displaystyle\int_{\R^{n}}|z| |(\p^{\alpha}\chi)(z)|\,dz\\
&\leq & C \lambda^{\alpha(|\gamma|-1)}.
\end{array}$$
This completes the proof of the Lemma.
\end{proof}
The coming statement claims the existence of particular solutions to the equation (\ref{equation n 2.2}). In the rest of this subsection, we will consider $A$ to be extended as $\tilde{A}$ outside $\Omega$. We denote by $A$ this extension.
\begin{Lemm}\label{Lemma 2.5} (see \cite{[BIB9]}) Given $\omega\in\mathbb{S}^{n-1}$ and  $\varphi\in \mathcal{C}^{\infty}_{0}(\R^{n})$.  Let $A\in W^{1,\infty}(\Omega)$ and $q\in L^{\infty}(\Omega)$. We consider the function $\phi$ defined by (\ref{equation  n 2.12}).  Then, for  any $\lambda>0$ the equation $\mathscr{H}_{A,q}u=0$ in $Q$ admits a solution 
   $$u\in \mathcal{C}([0,T];H^{1}(\Omega))\cap \mathcal{C}^{1}([0,T];L^{2}(\Omega)),$$
 of the form 
$$u(x,t)=\phi(x,t)\,b^{\sharp}_{\lambda}(x,t)\,e^{i\lambda(x\cdot \omega+t)}+r(x,t),$$
where 
$$b_{\lambda}^{\sharp}(x,t)=\exp \Big(i\int_{0}^{t} \omega\cdot A_{\lambda}^{\sharp} (x+s\omega)\,ds\Big).$$
Here  $A_{\lambda}^{\sharp}$ is given by (\ref{equation  n 2.13}) and the correction term $r$ satisfies
$$r(x,0)=\p_{t}r(x,0)=0,\,\,\mbox{in}\,\,\Omega,\quad r(x,t)=0\,\,\mbox{on}\,\,\Sigma.$$
Moreover, there exists a positive constant $C>0$ such that 
\begin{equation}\label{equation n 2.16}
\lambda^{\alpha} \|r\|_{L^{2}(Q)}+\lambda^{\alpha-1}\|\nabla r\|_{L^{2}(Q)}\leq C \|\varphi\|_{H^{3}(\R^{n})},\qquad 0< \alpha\leq 1/2.
\end{equation}
\end{Lemm}

\begin{proof}
In order to prove this lemma, it will be enough to show that if $r$ solves the following equation
\begin{equation}\label{equation n 2.17}
\left\{
  \begin{array}{ll}
 \mathscr{H}_{A,q}r=g, & \mbox{in} \,\,\,Q, \\
  \\
    r_{|t=0}=\p_{t}r_{|t=0}=0, & \mbox{in}\,\,\, \Omega,  \\
 \\  
    r=0, & \mbox{on} \,\,\,\Sigma,
  \end{array}
\right.
\end{equation}
then the estimate (\ref{equation n 2.16}) is satisfied.  Here the function $g$ is given by
$$\begin{array}{lll}
g(x,t)&=&-e^{i\lambda(x\cdot\omega+t)}(\p_{t}^{2}-\Delta_{A}+q(x))(\phi\,b^{\sharp}_{\lambda})(x,t)\\

&&- 2i\lambda e^{i\lambda(x\cdot \omega+t)}b^{\sharp}_{\lambda}(x,t)(\p_{t}-\omega\cdot\nabla )\phi(x,t)\\

&&-2i\lambda e^{i\lambda(x\cdot \omega+t)} \phi(x,t)(\p_{t}-\omega\cdot \nabla  -i(\omega\cdot A))b^{\sharp}_{\lambda}(x,t).
\end{array}$$
This and the fact that $\phi$ satisfies (\ref{equation  n 2.12}) and  $b^{\sharp}_{\lambda}$ solves the following equation 
$$(\p_{t}-\omega\cdot \nabla -i(\omega\cdot A^{\sharp}_{\lambda}))b^{\sharp}_{\lambda}=0,$$
immediately implies that 
$$g(x,t)=-e^{i\lambda(x\cdot\omega+t)}(\p_{t}^{2}-\Delta_{A}+q(x))(\phi\,b^{\sharp}_{\lambda})(x,t)+2\lambda e^{i\lambda(x\cdot \omega+t)}\omega\cdot (A^{\sharp}_{\lambda}-A)(x)(b^{\sharp}_{\lambda}\phi)(x,t),$$
with  $g\in L^{1}(0,T;L^{2}(\Omega))$. By setting  $w(x,t)=\displaystyle\int_{0}^{t}r(x,s)\,ds$, one can see that $w$ solves the hyperbolic problem (\ref{equation n 2.17}) with the right hand side
\begin{eqnarray}\label{equation n 2.18}
F(x,t)=\displaystyle\int_{0}^{t}\!\!\!\!\!\!\!&&\!\!\!\!\!\!\!g(x,s)\,ds=\displaystyle\int_{0}^{t}-e^{i\lambda(x\cdot \omega+s)}(\p_{t}^{2}-\Delta _{A}+q(x))(\phi \,b^{\sharp}_{\lambda})(x,s)\,ds\cr
&+&\displaystyle\int_{0}^{t}2\lambda e^{i\lambda(x\cdot \omega+s)}\omega\cdot (A^{\sharp}_{\lambda}-A)(x)(\phi\,b^{\sharp}_{\lambda})(x,s)\,ds
= F_{1}(x,t)+F_{2}(x,t).
\end{eqnarray}
Let us put  
\begin{equation}\label{equation n 2.19}
g_{1}(x,s)=(\p_{t}^{2}-\Delta_{A}+q(x))(\phi\,b^{\sharp}_{\lambda})(x,s),\quad \mbox{and}\quad g_{2}(x,s)=2\lambda\omega\cdot (A_{\lambda}^{\sharp}-A)(x)(\phi\,b^{\sharp}_{\lambda})(x,s).
\end{equation}
 In light of (\ref{equation n 2.18}) and (\ref{equation n 2.19}), we have 
 $$F_{1}(x,t)=-\frac{1}{i\lambda}\int_{0}^{t}g_{1}(x,s)\p_{s}(e^{i\lambda(x\cdot \omega+s)})\,ds.$$
 Thus, by integrating by parts with respect to $s$ and using (\ref{equation  n 2.15}), we get
\begin{equation}\label{equation n 2.20}
\|F_{1}\|^{2}_{L^{2}(Q)}
\leq\displaystyle\frac{C}{\lambda^{2}}\Big(\|g_{1}\|_{L^{2}(Q)}^{2}+T\|g_{1}(\cdot,0)\|_{L^{2}(\Omega)}^{2}+T \|\p_{s}g_{1}\|^{2}_{L^{2}(Q)}   \Big)\leq \displaystyle\frac{C}{\lambda^{2-\alpha}}\|\varphi\|_{H^{3}(\R^{n})}^{2}.
\end{equation}
On the other hand, in view of (\ref{equation n 2.18}) and (\ref{equation n 2.19}) we have 
$$F_{2}(x,t)=\frac{1}{i\lambda}\int_{0}^{t} g_{2}(x,s)\p_{s}(e^{i\lambda(x\cdot \omega+s)})\,ds. $$
Again, by integrating by parts with respect to the variable $s$, we find in view of (\ref{equation  n 2.14}) 
\begin{equation}\label{equation n 2.21}
\|F_{2}\|_{L^{2}(Q)}^{2}\leq \frac{C}{\lambda^{2\alpha}}\|\varphi\|^{2}_{H^{3}(\R^{n})}.
\end{equation}
Applying the standard energy estimate for hyperbolic initial boundary value problems to the solution $w$, we get from (\ref{equation n 2.20}) and (\ref{equation n 2.21})
$$\|r\|_{L^{2}(Q)}=\|\p_{t}w\|_{L^{2}(Q)}\leq \frac{C}{\lambda^{\alpha}}\|\varphi\|_{H^{3}(\R^{n})},\quad 0<\alpha\leq 1/2.$$
By using again the energy estimate applied to the solution $r$,  we get from (\ref{equation  n 2.14})
$$\begin{array}{lll}
\|\nabla  r\|_{L^{2}(Q)}\leq C \|g\|_{L^{2}(Q)}&\leq& C(\|g_{1}\|_{L^{2}(Q)}+\|g_{2}\|_{L^{2}(Q)})\\&\leq& C \Big(\lambda^{\alpha}+\displaystyle\frac{1}{\lambda^{\alpha-1}} \Big)\|\varphi\|_{H^{3}(\R^{n})}\\
&\leq& \displaystyle\frac{C}{ \lambda^{\alpha-1}}\|\varphi\|_{H^{3}(\R^{n})},
\end{array}$$
with $0<\alpha\leq 1/2$. This completes the proof of the Lemma.
\end{proof}
By a similar way, we can construct a solution to the backward problem.
\begin{Lemm}\label{Lemma 2.6}
  Given $\omega\in\mathbb{S}^{n-1}$ and  $\varphi\in \mathcal{C}^{\infty}_{0}(\R^{n})$. We consider the function $\phi$ defined by (\ref{equation  n 2.12}).  Then, for  any $\lambda>0$ the equation  $\mathscr{H}^{*}_{A,q}v=0$ in $Q$
 admits a solution 
   $$v\in \mathcal{C}([0,T];H^{1}(\Omega))\cap \mathcal{C}^{1}([0,T];L^{2}(\Omega)),$$
 of the form 
$$v(x,t)=\phi(x,t)\,b^{\sharp}_{\lambda}(x,t)\,e^{i\lambda(x\cdot \omega+t)}+r(x,t),$$
where 
$$b(x,t)=\exp \Big(i\int_{0}^{t} \omega\cdot \overline{A}^{\sharp}_{\lambda}(x+s\omega)\,ds\Big),$$
and $r(x,t)$ satisfies
$$r(x,T)=\p_{t}r(x,T)=0,\,\,\mbox{in}\,\,\Omega,\quad r(x,t)=0\,\,\mbox{on}\,\,\Sigma.$$
Moreover, there exists a positive constant $C>0$ such that 
\begin{equation}\label{equation N 2.22}
\lambda^{\alpha} \|r\|_{L^{2}(Q)}+\lambda^{\alpha -1}\|\nabla r\|_{L^{2}(Q)}\leq C \|\varphi\|_{H^{3}(\R^{n})}.
\end{equation}
\end{Lemm}

  
\subsection{Stability  for the magnetic field} In this section we are going to use the geometrical optics solutions constructed before in order to retrieve a stability  result for the determination of the magnetic field $d\alpha_{A}$ from the DN map $N_{A,q}$. Let us first consider $A_{1},A_{2}\in \mathcal{A}(M,A_{0})$ and $q_{1},\,q_{2}\in \mathcal{Q}(M,q_{0})$. We define 
$$A(x)=(A_{1}-A_{2})(x) \quad \mbox{and} \quad q(x)=(q_{2}-q_{1})(x).$$
Assume that there exists $\rho >0$ such that $T> \mbox{Diam}\, \Omega +4\rho$. We denote 
$$\mathcal{D}_{\rho}=\{x\in\R^{n}\setminus \Omega,\,\,\mbox{dist}(x,\Omega)<\rho\}.$$ 
Throughout the rest of the paper, we assume  that supp $\varphi\subset \mathcal{D}_{\rho}$,  so that  we have 
$$\mbox{supp}\,\varphi\cap \Omega=\emptyset,\quad\mbox{and}\quad (\mbox{supp}\,\varphi\pm T \omega)\cap \Omega=\emptyset.$$ 
we  recall that  $A$ is assumed to be extended as $\tilde{A}$ outside $\Omega$ and that we denoted by $A$ this extension. Moreover, we extend $q$ to a $L^{\infty}(\R^{n})$ function by defining it by zero outside $\Omega$. We denote by $A$ and $q$ these extensions. 
\subsubsection{Preliminary estimate}
 The main purpose of this subsection is to establish the following 
\begin{Lemm}\label{Lemma 2.7}
 There exists a constant $C>0$ such that for any $\omega\in\mathbb{S}^{n-1}$, the following estimate  
$$\Big|\int_{\R}\omega\cdot A(y-s\omega)\,ds\Big| \leq C \Big(
\lambda^{\delta}\|N_{A_{2},q_{2}}-N_{A_{1},q_{1}}\|+\frac{1}{\lambda^{\beta}}
\Big)\quad\mbox{a.e,}\,\color{black} y\in \R^{n}\color{black},$$
holds true for any \color{black} $\lambda>0$ sufficiently large. Here $C$ depends only on $\Omega$, $T$ and $M$. 
\end{Lemm}

\begin{proof}
In view of Lemma \ref{Lemma 2.5}, and using the fact that $\mbox{supp}\,\varphi\cap \Omega=\emptyset$, there exists a geometrical optic solution $u_{2}$ to the follwoing equation 
 \begin{equation*}
\left\{
  \begin{array}{ll}
\mathscr{H}_{A_{2},q_{2}}u_{2}=0, & \mbox{in} \,\,\,Q, \\
   \\
    u_{2|t=0}=\p_{t}u_{2|t=0}=0, & \mbox{in}\,\,\, \Omega,  \\
  \end{array}
\right.
\end{equation*}
in the following form
$$u_{2}(x,t)=\varphi_{2}(x+t\omega)b_{2,\lambda}^{\sharp}(x,t)e^{i\lambda(x\cdot \omega+t)}+r_{2}(x,t),$$
with $b_{2,\lambda}^{\sharp}(x,t)=\exp\Big(i\displaystyle\int_{0}^{t}\omega\cdot A_{2,\lambda}^{\sharp}(x+s\omega)\,ds \Big)$ and $r_{2}$ satisfies (\ref{equation n 2.16}).   Next, let us denote by $f_{\lambda}=u_{2|\Sigma}$. Let $u_{1}$ be a solution to the follwing system
\begin{equation*}
\left\{
  \begin{array}{ll}
 \mathscr{H}_{A_{1},q_{1}}u_{1}=0, & \mbox{in} \,\,\,Q, \\
  \\
    u_{1|t=0}=\p_{t}u_{1|t=0}=0, & \mbox{in}\,\,\, \Omega,  \\
   \\
    u_{1}=u_{2}:=f_{\lambda}, & \mbox{on} \,\,\,\Sigma.
  \end{array}
\right.
\end{equation*}
Putting $u=u_{1}-u_{2}$. Then, $u$ is a solution to 
\begin{equation}
\label{equation N 2.23}
\left\{
  \begin{array}{ll}
 \mathscr{H}_{A_{1},q_{1}}u=2iA\cdot \nabla u_{2}+V_{A} u_{2}+q u_{2}, & \mbox{in} \,\,\,Q, \\
   \\
    u_{|t=0}=\p_{t}u_{|t=0}=0, & \mbox{in}\,\,\, \Omega,  \\
    \\
    u=0, & \mbox{on} \,\,\,\Sigma,
  \end{array}
\right.
\end{equation} 
where $A=A_{1}-A_{2},$  $q=q_{2}-q_{1}$ and $V_{A}=i \,\mbox{div}A+(A_{2}^{2}-A_{1}^{2})$. On the other hand, Lemma \ref{Lemma 2.6} and the fact that $(\mbox{supp} \varphi \pm T\omega)\cap \Omega=\emptyset$, guarantee the existence of a geometrical optic solution $v$ to
\begin{equation*}
\left\{
  \begin{array}{ll}
 \mathscr{H}^{*}_{A_{1},q_{1}}v=0, & \mbox{in} \,\,\,Q, \\
\\
    v_{|t=T}=\p_{t}v_{|t=T}=0, & \mbox{in}\,\,\, \Omega,  \\
  \end{array}
\right.
\end{equation*} 
in the following form 
$$v(x,t)=\varphi_{1}(x+t\omega) b_{1,\lambda}^{\sharp}(x,t)e^{i\lambda(x\cdot\omega+t)}+r_{1}(x,t),$$
where $b_{1}(x,t)=\exp\big( i\displaystyle\int_{0}^{t}\omega\cdot \overline{A}^{\sharp}_{1,\lambda}(x+s\omega)\,ds\Big)$ and $r_{1}$ satisfies (\ref{equation N 2.22}). Multiplying the first equation in (\ref{equation N 2.23}) by $v$ and integrating by parts we get in view of (\ref{equation  n 2.3}),
\begin{eqnarray}\label{equation N 2.24}
\int_{Q}2iA\cdot \nabla u_{2}(x,t)\overline{v}(x,t)\,dx\,dt&=&\int_{\Sigma}(N_{A_{1},q_{1}}-N_{A_{2},q_{2}})(f) \overline{v}(x,t)\,d\sigma\,dt \cr
&&-\int_{Q}(V_{A}(x)+q(x))u_{2}(x,t)\overline{v}(x,t)\,dx\,dt.
\end{eqnarray} 
On the other hand, by replacing $u_{2}$ and $v$ by their expressions, we get 
\begin{eqnarray}\label{equation n 2.25}
\displaystyle\int_{Q}2iA\cdot \nabla u_{2}(x,t)\overline{v}(x,t)\,dx\,dt&=&\displaystyle\int_{Q}2i A\cdot \nabla (\phi_{2}b_{2,\lambda}^{\sharp})(x,t) (\overline{\phi_{1}b_{1,\lambda}^{\sharp}})(x,t)\,dx\,dt\cr
&&+\displaystyle\int_{Q}2iA\cdot \nabla (\phi_{2}b_{2,\lambda}^{\sharp})(x,t)\overline{r}_{1}(x,t) e^{i\lambda(x\cdot\omega+t)}\,dx\,dt\cr
&&
-2\lambda \displaystyle\int_{Q}\omega\cdot A(x) (\phi_{2}b_{2,\lambda}^{\sharp})(x,t) (\overline{\phi_{1}b_{1,\lambda}^{\sharp}})(x,t)\,dx\,dt\cr
&&-2\lambda \displaystyle\int_{Q} \omega\cdot A(x) (\phi_{2}b_{2,\lambda}^{\sharp})(x,t)\overline{r}_{1}(x,t) e^{i\lambda(x\cdot \omega+t)}\,dx\,dt\cr
&&+2i\displaystyle\int_{Q}A\cdot \nabla r_{2}(x,t)(\overline{\phi_{1}b_{1,\lambda}^{\sharp}})(x,t) e^{-i\lambda(x\cdot \omega+t)}\,dx\,dt\cr
&&+2i\displaystyle\int_{Q}A\cdot \nabla r_{2}(x,t) \overline{r}_{1}(x,t)\,dx\,dt\cr
&=&-2\lambda\displaystyle\int_{Q}\omega\cdot A(x) (\phi_{2}\overline{\phi}_{1})(x,t) (b_{2,\lambda}^{\sharp}\overline{b_{1}^{\sharp}}_{\lambda})(x,t) \,dx\,dt+I_{\lambda}. 
\end{eqnarray}
Using the fact that for $\lambda$ sufficiently large, we have  
\begin{equation}\label{equation n 2.26}
\|u_{2}\overline{v}\|_{L^{1}(Q)}\leq C \|\varphi_{1}\|_{H^{3}(\R^{n})}\|\varphi_{2}\|_{H^{3}(\R^{n})},\quad \mbox{and}\quad |I_{\lambda}|\leq C \lambda^{1-\alpha}\|\varphi_{1}\|_{H^{3}(\R^{n})}\|\varphi_{2}\|_{H^{3}(\R^{n})}.
\end{equation}
On the other hand from the trace theorem we have 
\begin{eqnarray*}
\Big|\displaystyle\int_{\Sigma}(N_{A_{2},q_{2}}-N_{A_{1},q_{1}})(f_{\lambda})\,\overline{v}(x,t)\,d\sigma\,dt\Big|&\leq&
\|N_{A_{2},q_{2}}-N_{A_{1},q_{1}}\|\,\|f_{\lambda}\|_{H^{1}(\Sigma)}\|\overline{v}\|_{L^{2}(\Sigma)}\cr
\cr
&\leq& \|N_{A_{2},q_{2}}-N_{A_{1},q_{1}}\|\,\|u_{2}-r_{2}\|_{H^{2}(Q)}\,\|\overline{v}-\overline{r}_{1}\|_{H^{1}(Q)}\cr
\cr
&\leq& C\lambda^{3}\|N_{A_{2},q_{2}}-N_{A_{1},q_{1}}\|\,\|\varphi_{1}\|_{H^{3}(\R^{n})}\|\varphi_{2}\|_{H^{3}(\R^{n})}.
\end{eqnarray*}
This, (\ref{equation n 2.25}) and (\ref{equation n 2.26}) yield 
\begin{eqnarray*}
\Big|\displaystyle\int_{Q}\omega\cdot A(x) (\phi_{2}\overline{\phi}_{1})(x,t) (b_{2,\lambda}^{\sharp}\!\!\!\!\!\!\!& & \!\!\!\!\!\!\!\!\!\overline{b_{1}^{\sharp}}_{\lambda})(x,t)dx\,dt\Big|\\
&\leq& C\Big(\lambda^{2}\|N_{A_{2},q_{2}}-N_{A_{1},q_{1}}\| +\frac{1}{\lambda^{\alpha}} \Big)\|\varphi_{1}\|_{H^{3}(\R^{n})}\|\varphi_{2}\|_{H^{3}(\R^{n})}.
\end{eqnarray*}
Since \color{black} $A=0$ outside $\Omega$, \color{black} then putting $y=x+t\omega$ and  $s=t-s$, we get for $\overline{\varphi}_{1}=\varphi_{2}=\varphi\in \mathcal{C}^{\infty}_{0}(\mathcal{D}_{\rho}),$
\begin{eqnarray*}
\Big|\displaystyle\int_{0}^{T}\int_{\R^{n}}\omega\cdot A(y-t\omega) \varphi^{2}(y) \exp\Big(-i\displaystyle\int_{0}^{t}\omega\cdot A^{\sharp}_{\lambda}\!\!\!\!\!\!\!& & \!\!\!\!\!\!\!\!\!(y-s\omega)\,ds \Big) \,dy\,dt\Big|\cr
&\leq& C\Big(\lambda^{2}\|N_{A_{2},q_{2}}-N_{A_{1},q_{1}}\| +\frac{1}{\lambda^{\alpha}} \Big)\|\varphi\|_{H^{3}(\R^{n})}^{2}.
\end{eqnarray*}
Now, using the fact that 
$$\begin{array}{lll}
\displaystyle\int_{0}^{T}\int_{\R^{n}}\omega\cdot A(y-t\omega)\varphi^{2}(y)\!\!\!\!\!&&\!\!\!\!\!\!\displaystyle\exp\Big(-i\int_{0}^{t}\omega\cdot A^{\sharp}_{\lambda}(y-s\omega)\,ds \Big)dy\,dt\\
&=&\!\!\displaystyle\int_{0}^{T}\int_{\R^{n}}\omega\cdot (A-A^{\sharp}_{\lambda})(y-t\omega)\varphi^{2}(y)\exp\Big(-i\int_{0}^{t}\omega\cdot A^{\sharp}_{\lambda}(y-s\omega)ds \Big)dy\,dt\\
&&\!\!\!+\displaystyle\int_{0}^{T}\int_{\R^{n}}\omega\cdot A^{\sharp}_{\lambda}(y-t\omega)\varphi^{2}(y)\displaystyle\exp\Big(-i\int_{0}^{t}\omega\cdot A^{\sharp}_{\lambda}(y-s\omega)\,ds \Big)dy\,dt.
\end{array}$$
we get from  (\ref{equation  n 2.14}) 
$$\Big|\int_{0}^{T}\int_{\R^{n}}\omega\cdot A^{\sharp}_{\lambda}(y-t\omega)\varphi^{2}(y)\displaystyle\exp\Big(-i\int_{0}^{t}\omega\cdot A^{\sharp}_{\lambda}(y-s\omega)\,ds \Big)dy\,dt  \Big|\leq C\Big(\lambda^{2}\|N_{A_{2},q_{2}}-N_{A_{1},q_{1}}\| +\frac{1}{\lambda^{\alpha}} \Big)\|\varphi\|_{H^{3}(\R^{n})}^{2}.$$
Therefore, since   
$$\p_{t}\Big[\exp \Big(-i\int_{0}^{t}\omega\cdot A^{\sharp}_{\lambda}(y-s\omega)\,ds  \Big)\Big]=-i\omega\cdot A^{\sharp}_{\lambda}(y-t\omega) \exp \Big(-i\int_{0}^{t}\omega\cdot A^{\sharp}_{\lambda}(y-s\omega)\,ds  \Big),  $$
we obtain the following estimation
\begin{equation}\label{equation n 2.27}
\Big|i\int_{\R^{n}}\varphi^{2}(y) \Big[\exp \Big( -i\int_{0}^{T} \omega\cdot A_{\lambda}^{\sharp}(y-s\omega)\Big)-1\Big]\,dy  \Big|\leq C\Big(  \lambda^{2} 
\|N_{A_{2},q_{2}}-N_{A_{1},q_{1}}\| +\frac{1}{\lambda^{\alpha}}\Big)\|\varphi\|_{H^{3}(\R^{n})}^{2}.
\end{equation}
We move now to specify the choice of the function $\varphi\in \mathcal{C}^{\infty}_{0}(\mathcal{D}_{\rho})$. We  set $B(0,r):= \{x\in\R^{n};\,\,|x|<r\}$ for all $r\geq 0$. Let $\psi\in\mathcal{C}_{0}^{\infty}(\R^{n})$ be a non-negative function which is
supported in the unit ball $B(0,1)$ and such that
$\|\psi\|_{L^{2}(\R^{n})}=1$. For $y\in\mathcal{D}_{\rho}$, we define
\begin{equation}\label{equation n 2.28}
\varphi_{h}(x)=h^{-n/2}\psi\Big(\frac{x-y}{h}\Big).
\end{equation} 
 Then, for $h>0$ sufficiently small such that Supp $\varphi_{h}\subset \mathcal{D}_{\rho}$.  We can verify that
$$
\mbox{Supp}\,\varphi_{h}\cap\Omega=\displaystyle\emptyset\quad \mbox{and} \quad(\mbox{Supp}\,\varphi_{h}\pm T\omega)\cap \Omega=\displaystyle\emptyset.
$$
Moreover, we have
\begin{eqnarray}\label{equation n 2.29}
\!\!\!\!\!\!\!\!\!&\Big|&\!\!\!\exp\Big(-i\displaystyle\int_{0}^{T}\omega\cdot A^{\sharp}_{\lambda}(y-s\omega)\,ds
\Big)-1 \Big|=\Big|
\displaystyle\int_{\R^{n}}\varphi_{h}^{2}(x)\Big[ \exp\Big(-i \int_{0}^{T}\omega\cdot A^{\sharp}_{\lambda}(y-s\omega)\,ds \Big)-1  \Big]\,dx\Big|\cr
&&\leq\quad\Big| \displaystyle\int_{\R^{n}}\varphi_{h}^{2}(x)\Big[\exp\Big( -i
\displaystyle\int_{0}^{T}\omega\cdot A^{\sharp}_{\lambda}(y-s\omega)\,ds \Big)
-\exp\Big(-\displaystyle i\int_{0}^{T}\omega\cdot A^{\sharp}_{\lambda}(x-s\omega)\,ds\Big)   \Big]  \,dx\Big|\cr
 &&\quad\quad\quad\quad\qquad\quad
 +\Big| \displaystyle\int_{\R^{n}} \varphi_{h}^{2}(x)\Big[\exp\Big(-\displaystyle i\int_{0}^{T}\omega\cdot A^{\sharp}_{\lambda}(x-s\omega)\,ds  \Big)-1 \Big]
 dx\Big|.
\end{eqnarray}
Using the fact that
$$
\Big|\displaystyle\int_{0}^{T}\!\! \Big( i\omega\cdot A^{\sharp}_{\lambda}(y-s\omega)\!-\!i\omega\cdot A^{\sharp}_{\lambda}(x-s\omega)\Big) ds\Big|\leq C \,|y-x|,
$$
 we deduce upon replacing $\varphi=\varphi_{h}$ in (\ref{equation n 2.27}), the following estimation
$$
\Big|\exp\Big(-i\int_{0}^{T}\omega\cdot A^{\sharp}_{\lambda}(y-s\omega)\,ds  \Big)-1
\Big|\leq \!C\!\int_{\R^{n}}\!\varphi_{h}^{2}(x)\,|y-x|\,dx+C\Big(
\lambda^{2}\|N_{A_{2},q_{2}}-N_{A_{1},q_{1}}\|+\frac{1}{\lambda^{\alpha}}
\Big)\|\varphi_{h}\|^{2}_{H^{3}(\R^{n})}.
$$
 On the other hand, we have
$$
\|\varphi_{h}\|_{H^{3}(\R^{n})}\leq C h^{-3}\quad\mbox{and}\quad\int_{\R^{n}}\varphi_{h}^{2}(x)|y-x|\,dx \leq C h.
$$
So, we end up getting the following inequality
$$
 \Big|\exp\Big(-i \int_{0}^{T}\omega\cdot A^{\sharp}_{\lambda}(y-s\omega)\,ds  \Big)-1
\Big|\leq C \,h+C\Big(
\lambda^{2}\|N_{A_{2},q_{2}}-N_{A_{1},q_{1}}\|+\frac{1}{\lambda^{\alpha}}
\Big)h^{-6}.
 $$
Selecting $h$ small such that $h=1/\lambda^{\alpha} h^{6}$,
that is $h=\lambda^{-\alpha/7}$, we find  $\delta>1$ and $0<\beta<\alpha<1$
such that
\begin{equation}\label{equation n 2.30}
\Big| \exp \Big( -i\int_{0}^{T}\omega\cdot A^{\sharp}_{\lambda}(y-s\omega)\,ds \Big)-1 \Big|\leq
C \Big(
\lambda^{\delta}\|N_{A_{2},q_{2}}-N_{A_{1},q_{1}}\|+\frac{1}{\lambda^{\beta}}
\Big).
\end{equation}
Using the fact that $|X|\leq e^{M}\,|e^{X}-1|$ for any $X$ real satisfying $|X|\leq M$  we  found out that 
$$\Big| -i\int_{0}^{T}\omega\cdot A^{\sharp}_{\lambda}(y-s\omega)\,ds \Big|\leq e^{MT}\Big| \exp\Big( -i\int_{0}^{T}\omega\cdot A(y-s\omega)\,ds  \Big)-1
\Big|,$$
where  $X=\displaystyle\int_{0}^{T}-i\omega\cdot A^{\sharp}_{\lambda}(y-s\omega)\,ds$. 
We conclude in light of (\ref{equation n 2.30}) the following the estimate
\begin{equation}\label{equation n 2.31}
\Big|\int_{0}^{T}\omega\cdot A^{\sharp}_{\lambda}(y-s\omega)\,ds \Big|\leq C \Big(
\lambda^{\delta}\|N_{A_{2},q_{2}}-N_{A_{1},q_{1}}\|+\frac{1}{\lambda^{\beta}}
\Big),\quad\mbox{a.\,e. }\,\,\,y\in \mathcal{D}_{\rho},\quad\omega\in\mathbb{S}^{n-1}.
\end{equation}
By replacing $\omega$ by $-\omega$, we get 
\begin{equation}\label{equation n 2.32}
\Big|\int_{-T}^{0}\omega\cdot A^{\sharp}_{\lambda}(y-s\omega)\,ds \Big| \leq C \Big(
\lambda^{\delta}\|N_{A_{2},q_{2}}-N_{A_{1},q_{1}}\|+\frac{1}{\lambda^{\beta}}
\Big),\quad\mbox{a.\,e.}\,\,\,y\in \mathcal{D}_{\rho},\quad\omega\in\mathbb{S}^{n-1}.
\end{equation}
Bearing in mind that 
$$\Big|\int_{-T}^{T}\omega\cdot A(y-s\omega)\,ds \Big|\leq \Big|\int_{-T}^{T}\omega\cdot A_{\lambda}^{\sharp}(y-s\omega)\,ds\Big|+\Big| \int_{-T}^{T}\omega\cdot (A_{\lambda}^{\sharp}-A)(y-s\omega)\,ds \Big|,$$
we can deduce from (\ref{equation n 2.32}) and (\ref{equation  n 2.14}) the following estimate 
$$\begin{array}{lll}
\displaystyle\Big|\int_{-T}^{T}\omega\cdot A(y-s\omega)\,ds\Big|&\leq& C \Big(
\lambda^{\delta}\|N_{A_{2},q_{2}}-N_{A_{1},q_{1}}\|+\displaystyle\frac{1}{\lambda^{\beta}}+\displaystyle\frac{1}{\lambda^{\alpha}}\Big)\\
&\leq& C\Big( \lambda^{\delta}\|N_{A_{2},q_{2}}-N_{A_{1},q_{1}}\|+\displaystyle\frac{1}{\lambda^{\beta}}
\Big),
\end{array}$$
for all $y\in \mathcal{D}_{\rho}$ and $\omega\in\mathbb{S}^{n-1}$. Since  $T>\mbox{Diam}\,\Omega+4\rho$ and \color{black} supp $A\subseteq \Omega$ we obtain in view of (\ref{equation n 2.31})-(\ref{equation n 2.32}), 
$$\Big|\int_{\R}\omega\cdot A(y-s\omega)\,ds\Big| \leq C \Big(
\lambda^{\delta}\|N_{A_{2},q_{2}}-N_{A_{1},q_{1}}\|+\frac{1}{\lambda^{\beta}}
\Big),\quad\mbox{a.\,e.}\,\,\,\color{black}y\in \R^{n}\color{black},\quad\omega\in\mathbb{S}^{n-1}.$$
This completes the proof of the Lemma.
\end{proof}

  
\subsubsection{An estimate for the magnetic field} \label{subsub1}
In this section we estimate the magnetic field $d\alpha_{A_{1}}-d\alpha_{A_{2}}$ by the use of the lemma proved in the previous section. For this purpose, let us first introduce this notation
$$a_{k}(x)=(A_{1}-A_{2})(x)\cdot e_{k}=A(x)\cdot e_{k},$$
where $(e_{k})_{1\leq k\leq n}$ is the canonical basis of $\R^{n}$. On the other hand, we denote  by 
\begin{equation}\label{equation n 2.33}
\sigma_{j,k}(x)=\frac{\p a_{k}}{\p x_{j}}(x)-\frac{\p a_{j}}{\p x_{k}}(x), \quad j,k=1...n.
\end{equation}
Let $\xi\in\omega^{\perp}$. By the change of variables $x=z-s \omega\in \omega^{\perp}\oplus \R \omega=\R^{n}$, we have  the following identity
$$z \cdot \xi=z \cdot \xi-s \omega\cdot \xi=x\cdot \xi,$$ with $d x =d \sigma dt$. Thus,  we get 
\begin{eqnarray*}
\displaystyle\int_{\omega^{\perp}}e^{-iz\cdot \xi}\int_{\R}\omega\cdot A(z-s\omega)\,ds\,d\sigma&=&
 \displaystyle\int_{\omega^{\perp}}\int_{\R}e^{-iz\cdot\xi+s\xi\cdot \omega } (\omega\cdot A)(z-s\omega)\,ds\,d\sigma\cr
 &=& \displaystyle\int_{\R^{n}}e^{-ix\cdot \xi}\omega\cdot A(x)\,dx.
 \end{eqnarray*}
 Assume that $\Omega\subset{B}(0, R_{1})$, with $R>0$. Using the fact that Supp $A\subset \Omega$, we get from Lemma \ref{Lemma 2.7} 
 \begin{eqnarray}\label{equation n 2.34}
 \Big|  \displaystyle\int_{\R^{n}}e^{-ix\cdot \xi}\omega\cdot A(x)\,dx \Big|&\leq& \displaystyle \int_{\omega^{\perp}\cap B(0,R_{1})}   \Big| e^{-iz\cdot \xi}\displaystyle \int_{\R}\omega\cdot A(z-s\omega)\,ds\Big| \,\,dz\cr
 \cr
 &\leq& C \,\Big(
\lambda^{\delta}\|N_{A_{2},q_{2}}-N_{A_{1},q_{1}}\|+\displaystyle\frac{1}{\lambda^{\beta}}
\Big).
 \end{eqnarray}
For $\xi\in\R^{n}$, we define $\omega=\frac{\xi_{j}e_{k}-\xi_{k}e_{j}}{|\xi_{j}e_{k}-\xi_{k}e_{j}|}$. Multiplying (\ref{equation n 2.34}) by $|\xi_{j}e_{k}-\xi_{k}e_{j}|$, we obtain
 $$\Big|\int_{\R^{n}}e^{-ix\cdot\xi }\Big(\xi_{j}a_{k}(x)-\xi_{k}a_{j}(x)  \Big)\,dx\Big|\leq C |\xi_{j}e_{k}-\xi_{k}e_{j}|\Big(
\lambda^{\delta}\|N_{A_{2},q_{2}}-N_{A_{1},q_{1}}\|+\displaystyle\frac{1}{\lambda^{\beta}}
\Big).$$
This together with  (\ref{equation n 2.33}) yield
$$|\widehat{\sigma}_{j,k}(\xi)|\leq C \color{black} <\xi>\color{black} \Big(\lambda^{\delta}\|N_{A_{2},q_{2}}-N_{A_{1},q_{1}}\|+\displaystyle\frac{1}{\lambda^{\beta}}
   \Big).$$ 
We are now in position to upper bound the magnetic field induced by the magnetic potential in suitable norms. For this purpose, let $0<R\leq \lambda$. In light of the above reasoning, this can be achieved by  decomposing the $H^{-1}(\R^{n})$ norm of $\sigma_{j,k}$ as follows
$$
\|\sigma_{j,k}\|^{2}_{H^{-1}(\R^{n})}=\displaystyle\int_{|\xi|\leq
R}|\widehat{\sigma}_{j,k}(\xi)|^{2}<\xi>^{-2}\,d\xi
+\displaystyle\int_{|\xi|>R}|\widehat{\sigma}_{j,k}(\xi)|^{2}
<\xi>^{-2}\,d\xi.$$
Then, we have
$$\|\sigma_{j,k}\|^{2}_{H^{-1}(\R^{n})}\leq C \Big(R^{n}\|<\xi>^{-1}\widehat{\sigma}_{j,k}\|^{2}_{L^{\infty}(B(0,R))}+\frac{1}{R^{2}}
\|\sigma_{j,k}\|_{L^{2}(\R^{n})}^{2}\Big).$$
This entails that
$$\|\sigma_{j,k}\|^{2}_{H^{-1}(\R^{n})}\leq C \Big( R^{n}\para{\lambda^{2\delta}\|N_{A_{2},q_{2}}
-N_{A_{1},q_{1}}\|^{2}+\frac{1}{\lambda^{2\beta}}}+\frac{1}{R^{2}}\Big).
$$
Next, we choose  $R>0$ in such away $R^{n}/\lambda^{2\beta}=1/R^{2}$. Thus, we find  $\mu_{1}>2$ and $\mu_{2}>0$ such that
\begin{eqnarray}\label{equation n 2.35}
\|\sigma_{j,k}\|^{2}_{H^{-1}(\R^{n})}&\leq& C\Big(\lambda^{\frac{2\beta}{n+2}+2\delta}\|N_{A_{2},q_{2}}-N_{A_{1},q_{1}}\|^{2}+\lambda^{\frac{2\beta n}{n+2}-2\beta}\Big)\cr
&\leq& C \Big(\lambda^{\mu_{1}}\|N_{A_{2},q_{2}}-N_{A_{1},q_{1}}\|^{2}+\lambda^{-\mu_{2}}\Big).
\end{eqnarray}
Now we assume that $\|N_{A_{2},q_{2}}-N_{A_{1},q_{1}}\|<c<1$,  and we minimize with respect to $\lambda$
 to end up getting  
$$\|\sigma_{j,k}\|_{H^{-1}(\R^{n})}\leq C
\|N_{A_{2},q_{2}}-N_{A_{1},q_{1}}\|^{1/2}.$$
The above estimate remains true in the case where 
 $\|N_{A_{2},q_{2}}-N_{A_{1},q_{1}}\|\geq c$, since we have 
$$\|\sigma_{j,k}\|_{H^{-1}(\R^{n})}\leq \frac{2 M}{c^{1/2}}c^{1/2}\leq \frac{2M}{c^{1/2}}\|N_{A_{2},q_{2}}-N_{A_{1},q_{1}}\|^{1/2}. $$
Therefore, we find out that 
\begin{eqnarray}\label{equation n 2.36}
\|d\alpha_{A_{1}}-d\alpha_{A_{2}}\|_{H^{-1}(\Omega)}\leq\displaystyle\sum_{j,k}\|\sigma_{j,k}\|_{H^{-1}(\R^{n})}
\leq C\|N_{A_{2},q_{2}}-N_{A_{1},q_{1}}\|^{1/2}.
\end{eqnarray}

  
\subsection{Stability for the electric potential}\label{subsub2}
The goal of this section is to prove a stability estimate for the electric potential. The proof involves using the stability estimate we have already obtained for the magnetic field. We will proceed as in \cite{[BIB24]}.

 Let $n<p_{0}<\infty$. Apply the Hodge decomposition to $A_{1}-A_{2}$ in the space $W^{1,p_{0}}(\Omega,\mathbb{C}^{n})$. We  define 
\begin{equation}
\label{equation n 2.37}
A^{'}_{1}=A_{1}+\frac{1}{2}\nabla \psi,\quad\mbox{and}\quad  A^{'}_{2}=A_{2}-\frac{1}{2}\nabla\psi,
\end{equation}
 with $\psi\in W^{3,p_{0}}(\Omega)\cap H^{1}_{0}(\Omega)$.  From Lemma 6.2 given in \cite{[BIB24]},  $A'=A^{'}_{2}-A'_{1} $ satisfies 
\begin{equation}\label{equation n 2.38}
\|A'\|_{W^{1,p_{0}}(\Omega)}\leq C \|d\alpha_{A_{1}}-d\alpha_{A_{2}}\|_{L^{p_{0}}(\Omega)}=C \|d\alpha_{A'_{1}}-d\alpha_{A'_{2}}\|_{L^{p_{0}}(\Omega)}.
\end{equation}
Recall that  since the DN map is invariant under gauge transformation then  we have 
\begin{equation}\label{equation n 2.39}
N_{A_{1},q_{1}}=N_{A_{1}+\frac{1}{2}\nabla \psi,q_{1}},\quad N_{A_{2},q_{2}}=N_{A_{2}-\frac{1}{2}\nabla \psi,q_{2}}.
\end{equation}

Throughout the rest of this section, $A_{j}$ will be replaced by $A'_{j}$ for  $j=1,\,2$.
\subsubsection{Preliminary estimate}
\begin{Lemm}\label{Lemma 2.8}There exist a constant $C>0$ such that for any $\omega\in\mathbb{S}^{n-1}$, the following estimate
$$
 \Big| \int_{\R}q(y-t\omega)\,dt \Big|\leq  C
 \Big(\lambda^{\delta}\|N_{A_{2},q_{2}}-N_{A_{1},q_{1}}\|+\frac{1}{\lambda^{\beta}}    \Big),\quad \mbox{a.\,e.\,} \,y\in\R^{n},
$$
holds true. Here $C$ depends only on $\Omega$, $T$ and $M$.
\end{Lemm}
\begin{proof}
We start with the identity (\ref{equation N 2.24}) except this time we isolate the electric potential term
$$\begin{array}{lll}
\displaystyle\int_{Q}\!\!\!&q(x)&\!\!\!u_{2}(x,t)\overline{v}(x,t)\,dx\,dt=\displaystyle\int_{\Sigma}(N_{A'_{1},q_{1}}-N_{A'_{2},q_{2}})(f) \overline{v}(x,t)\,d\sigma\,dt \\
&&-\displaystyle\int_{Q}2iA'\cdot \nabla u_{2}(x,t)\overline{v}(x,t)\,dx\,dt-\displaystyle\int_{Q}V_{A'}(x)u_{2}(x,t)\overline{v}(x,t)\,dx\,dt,
\end{array}$$
where $V_{A'}=i\mbox{div}A'+(A_{2}^{'2}-A_{1}^{'2})$. By replacing $u_{2}$ and $v$ by their expressions, we get 
$$\begin{array}{lll}
\displaystyle\int_{Q}q(x)u_{2}(x,t)\overline{v}(x,t)\,dx\,dt&=&\displaystyle\int_{Q}q(x)(\phi_{2}\overline{\phi}_{1})(x,t)(b_{2,\lambda}^{\sharp}\overline{b}_{1,\lambda}^{\sharp})(x,t)\,dx\,dt\\
&&+\displaystyle\int_{Q}q(x) \phi_{2}(x,t)b_{2,\lambda}^{\sharp}(x,t)e^{i\lambda(x\cdot\omega+t)}\overline{r}_{1}(x,t)\,dx\,dt\\
&&+\displaystyle\int_{Q}q(x) \overline{\phi}_{1}(x,t)\overline{b_{1}^{\sharp}}_{\lambda}(x,t)e^{-i\lambda(x\cdot\omega+t)}r_{2}(x,t)\,dx\,dt\\
&&+\displaystyle\int_{Q}q(x)r_{2}(x,t)\,\overline{r}_{1}(x,t)\,dx\,dt.
\end{array}$$
Therefore, we have the following identity 
\begin{equation}\label{equation n 2.40}
\int_{Q}q(x)(\phi_{2}\overline{\phi}_{1})(x,t)(b_{2,\lambda}^{\sharp}\overline{b_{1}^{\sharp}}_{\lambda})(x,t)\,dx\,dt=\displaystyle\int_{\Sigma}(N_{A'_{1},q_{1}}-N_{A'_{2},q_{2}})(f) \overline{v}(x,t)\,d\sigma\,dt+I_{\lambda},
\end{equation}
where $I_{\lambda}$ is given by 
$$\begin{array}{lll}
I_{\lambda}&=&\displaystyle\int_{Q}2iA'\cdot \nabla u_{2}(x,t)\overline{v}(x,t)\,dx\,dt-\displaystyle\int_{Q}V_{A'}(x)u_{2}(x,t)\overline{v}(x,t)\,dx\,dt\\
&&-\displaystyle\int_{Q}q(x)\phi_{2}(x,t)b_{2,\lambda}^{\sharp}(x,t)e^{i\lambda(x\cdot \omega+t)}\overline{r}_{1}(x,t)\,dx\,dt-\displaystyle\int_{Q}q(x) r_{2}(x,t) \overline{r}_{1}(x,t)\,dx\,dt
\\
&&-\displaystyle\int_{Q}q(x)\overline{\phi}_{1}(x,t) \overline{b_{1}^{\sharp}}_{\lambda}(x,t)e^{-i\lambda(x\cdot \omega+t)}r_{2}(x,t)\,dx\,dt.
\end{array}$$
For $\lambda$ sufficiently large, we have
\begin{equation}\label{equation n 2.41}
|I_{\lambda}|\leq C \Big(\frac{1}{\lambda^{\alpha}}+\lambda \|A'\|_{L^{\infty}(\Omega)} \Big)\|\varphi_{1}\|_{H^{3}(\R^{n})}\|\varphi_{2}\|_{H^{3}(\R^{n})},
\end{equation}
with $0<\alpha\leq 1/2$. On the other hand, by the trace theorem, we have 
\begin{equation}\label{equation n 2.42}
\Big| \displaystyle\int_{\Sigma}(N_{A_{1},q_{1}}-N_{A_{2},q_{2}})(f) \overline{v}(x,t)\,d\sigma\,dt   \Big|\leq C\lambda^{3}\|N_{A_{1},q_{1}}-N_{A_{2},q_{2}}\|\|\varphi_{1}\|_{H^{3}(\R^{n})}\|\varphi_{2}\|_{H^{3}(\R^{n})}.
\end{equation}
Thus, from (\ref{equation n 2.40}), (\ref{equation n 2.41}) and (\ref{equation n 2.42}) we obtain 
 \begin{eqnarray*}
 \Big| \int_{Q}\!\!\!q(x)(\phi_{2}\overline{\phi}_{1})\!\!\!\!\!&&\!\!\!\!\!\!(x,t)(b_{2,\lambda}^{\sharp}\overline{b}_{1,\lambda}^{\sharp})(x,t)\,dx\,dt \Big|\\
 &&\leq C\Big( \lambda^{3}\|N_{A_{1},q_{1}}-N_{A_{2},q_{2}}\|+\lambda\|A'\|_{L^{\infty}(\Omega)}+\frac{1}{\lambda^{\alpha}}  \Big)\|\varphi_{1}\|_{H^{3}(\R^{n})}\|\varphi_{2}\|_{H^{3}(\R^{n})}.
 \end{eqnarray*}
Since $q=0$ outside $\Omega$, then by the change of variables $y=x+t\omega$, $s=t-s$ we get for $\overline{\varphi}_{1}=\varphi_{2}=\varphi$, 
$$\begin{array}{lll}
\Big|\displaystyle\int_{0}^{T}\int_{\R^{n}}q(y-t\omega)\varphi^{2}(y)\!\!\!\!\!\!\!&&\!\!\!\!\!\!\exp\Big(-i\displaystyle\int_{0}^{t}\omega\cdot A'^{\sharp}_{\lambda}(y-s\omega)\,ds  \Big)\,dy\,dt   \Big|\\
&\leq&  C\Big( \lambda^{3}\|N_{A_{1},q_{1}}-N_{A_{2},q_{2}}\|+\lambda\|A'\|_{L^{\infty}(\Omega)}+\displaystyle\frac{1}{\lambda^{\alpha}}  \Big)\|\varphi\|^{2}_{H^{3}(\R^{n})},
\end{array}$$
with $A'^{\sharp}_{\lambda}=\chi_{\lambda}\ast A'$. This and the fact that 
$$\begin{array}{lll}
\Big|\displaystyle\int_{0}^{T}\int_{\R^{n}}q(y-t\omega)\varphi^{2}(y)\Big[1-\!\!\!\!&&\!\!\!\!\!\!\exp\Big(-i\displaystyle \int_{0}^{t}\omega\cdot A'^{\sharp}_{\lambda}(y-s\omega)\,ds   \Big)  \Big]  \,dy\,dt  \Big|\\
&&
\leq C\|A'^{\sharp}_{\lambda}\|_{L^{\infty}(\R^{n})}\|\varphi\|^{2}_{H^{3}(\R^{n}))}\leq C\|A'\|_{L^{\infty}(\Omega)}\|\varphi\|_{H^{3}(\R^{n})}^{2},
\end{array}$$
implies that  
$$\Big|\int_{0}^{T}\int_{\R^{n}}q(y-t\omega)\varphi^{2}(y)\,dy\,dt \Big|\leq \Big( \lambda^{3}\|N_{A_{1},q_{1}}-N_{A_{2},q_{2}}\|+\lambda\|A'\|_{L^{\infty}(\Omega)}+\displaystyle\frac{1}{\lambda^{\alpha}}  \Big)\|\varphi\|^{2}_{H^{3}(\R^{n})}.$$
Applying Morrey's inequality given by the following estimate 
$$\|A'\|_{\mathcal{C}^{0,1-\frac{n}{p}}(\Omega)}\leq C \|A'\|_{W^{1,p}(\Omega)},\quad A'\in W^{1,p}(\Omega),$$
where $n<p\leq \infty$ and $C$ a positive constant which depends on $p$, $n$ and $\Omega$, we get 
$$\Big|\int_{0}^{T}\int_{\R^{n}}q(y-t\omega)\varphi^{2}(y)\,dy\,dt \Big|\leq \Big( \lambda^{3}\|N_{A_{1},q_{1}}-N_{A_{2},q_{2}}\|+\lambda\|A'\|_{W^{1,p_{0}}(\Omega)}+\displaystyle\frac{1}{\lambda^{\alpha}}  \Big)\|\varphi\|^{2}_{H^{3}(\R^{n})},$$
where $n<p_{0}<\infty$. Hence, in light of (\ref{equation n 2.38}), we find out that
\begin{equation}\label{equation n 2.43}
\Big|\int_{0}^{T}\int_{\R^{n}}q(y-t\omega)\varphi^{2}(y)\,dy\,dt \Big|\leq \Big( \lambda^{3}\|N_{A_{1},q_{1}}-N_{A_{2},q_{2}}\|+\lambda\|d\alpha_{A_{1}}-d\alpha_{A_{2}}\|_{L^{p_{0}}(\Omega)}+\displaystyle\frac{1}{\lambda^{\alpha}}  \Big)\|\varphi\|^{2}_{H^{3}(\R^{n})}.
\end{equation}
By interpolating, we have for $ s=2/p_{0}$
$$\begin{array}{lll}
\|d\alpha_{A_{1}}-d\alpha_{A_{2}}\|_{L^{p_{0}}(\Omega)}&\leq& \|d\alpha_{A_{1}}-d\alpha_{A_{2}}\|_{L^{\infty}(\Omega)}^{1-s}\|d\alpha_{A_{1}}-d\alpha_{A_{2}}\|_{L^{2}(\Omega)}^{s}\\
&\leq& C \|d\alpha_{A_{1}}-d\alpha_{A_{2}}\|_{H^{1}(\Omega)}^{s/2}\|d\alpha_{A_{1}}-d\alpha_{A_{2}}\|^{s/2}_{H^{-1}(\Omega)}\\
&\leq& C \|d\alpha_{A_{1}}-d\alpha_{A_{2}}\|^{s/2}_{H^{-1}(\Omega)}.
\end{array}$$
Therefore, from (\ref{equation n 2.43}) and (\ref{equation n 2.36}), we obtain 
$$\begin{array}{lll}
\displaystyle\Big|\int_{0}^{T}\int_{\R^{n}}q(y-t\omega)\varphi^{2}(y)\,dy\,dt \Big|&\leq& \Big( \lambda^{3}\|N_{A_{1},q_{1}}-N_{A_{2},q_{2}}\|+\lambda \|N_{A_{1},q_{1}}-N_{A_{2},q_{2}}\|^{s/4}+\displaystyle\frac{1}{\lambda^{\alpha}}  \Big)\|\varphi\|^{2}_{H^{3}(\R^{n})}\\
&\leq& C  \Big( \lambda^{3}\|N_{A_{1},q_{1}}-N_{A_{2},q_{2}}\|+\displaystyle\frac{1}{\lambda^{\alpha}}  \Big)\|\varphi\|^{2}_{H^{3}(\R^{n})}.
\end{array}$$
Now we just need to proceed as in the determination of the magnetic field. We  consider  the
sequence  $(\varphi_{h})_{h}$ defined by (\ref{equation n 2.28}) with
$y\in\mathcal{D}_{\rho}$. Since  
$$\begin{array}{lll}
&&\Big|\displaystyle\int_{0}^{T}q(y-t\omega)dt\Big|=\Big|\displaystyle\int_{0}^{T}\int_{\R^{n}}q(y-t\omega)\,\varphi_{h}^{2}(x)\,dx\,dt
\Big|\\
&&\quad\qquad\qquad\quad \leq \Big| \displaystyle\int_{0}^{T}\!\!\!\int_{\R^{n}}\!\!q(x-t\omega)\varphi_{h}^{2}(x)dx\,dt\Big|
+\Big| \displaystyle\int_{0}^{T}\!\!\!\int_{\R^{n}}\!\!\!\Big(q(y-t\omega)-q(x-t\omega)\Big)\varphi_{h}^{2}(x)dx\,dt \Big|,
\end{array}$$
and using the fact that
$|q(y-t\omega)-q(x-t\omega)|\leq C |y-x|$, we obtain
$$\begin{array}{lll}
\Big|\displaystyle\int_{0}^{T}q(y-t\omega)dt\Big| \leq\!
C\Big(\lambda^{3}\|N_{A_{2},q_{2}}-N_{A_{1},q_{1}}\|+\displaystyle\frac{1}{\lambda^{\alpha}}
\Big)
\|\varphi_{h}\|^{2}_{H^{3}(\R^{n})}+C\!\!\displaystyle\int_{\R^{n}}\!\!|x-y|\varphi_{h}^{2}(x)dx.
\end{array} $$
On the other hand, since $\|\varphi_{h}\|_{H^{3}(\R^{n})}\leq C h^{-3}$
 and  $ \displaystyle\int_{\R^{n}} |x-y|
\varphi_{h}^{2}(x)\,dx\leq \,C \,h$, we conclude that
$$
\Big|\int_{0}^{T}q(y-t\omega)\,dt\Big|\leq C \Big( \lambda^{3}\|N_{A_{2},q_{2}}-N_{A_{1},q_{1}}\|
+\frac{1}{\lambda^{\alpha}}\Big)h^{-6}+C\,h.
$$ Selecting $h$ small such that  $h=h^{-6}/\lambda^{\alpha}$. Then, we find two
constants $\delta>1$ and $0<\beta<\alpha<1$ such that
\begin{equation}\label{equation n 2.44}
\Big|\int_{0}^{T}q(y-t\omega) \Big|\leq C
\Big(\lambda^{\delta}\|N_{A_{2},q_{2}}-N_{A_{1},q_{1}}\|
+\frac{1}{\lambda^{\beta}}    \Big).
\end{equation}
The estimate (\ref{equation n 2.44}) remains true by replacing $\omega$ by $-\omega$. Then we get for all $y\in \mathcal{D}_{\rho}$, 
$$\Big|\int_{-T}^{T}q(y-t\omega) \,dt\Big|\leq C
\Big(\lambda^{\delta}\|N_{A_{2},q_{2}}-N_{A_{1},q_{1}}\|+\frac{1}{\lambda^{\beta}}    \Big).$$
Next, using the fact that $q=q_{2}-q_{1}=0$ outside $\Omega$ and since $T>$Diam $\Omega+4\rho$, we have 
$$
 \Big| \int_{\R}q(y-t\omega)\,dt \Big|\leq  C
 \Big(\lambda^{\delta}\|N_{A_{2},q_{2}}-N_{A_{1},q_{1}}\|+\frac{1}{\lambda^{\beta}}    \Big),\quad \mbox{a.\,e.\,} \,y\in\R^{n},\,\,\omega\in\mathbb{S}^{n-1}.
$$
This completes the proof of the Lemma.
\end{proof}

  
\subsubsection{Estimate for the electric potential} This section is devoted to upper bound the electric potential. 
In light of Lemma \ref{Lemma 2.8} and arguing as in Section \ref{subsub1}, we get for all $\xi\in \omega^{\perp}$ the following estimate
\begin{equation}\label{equation n 2.45}
|\widehat{q}(\xi)|\leq C  \Big(\lambda^{\delta}\|N_{A_{2},q_{2}}-N_{A_{1},q_{1}}\|
+\frac{1}{\lambda^{\beta}}\Big).
 \end{equation}
By changing $\omega\in\mathbb{S}^{n-1}$ (\ref{equation n 2.45}) holds for all $\xi\in\R^{n}$.
 By decomposing the $H^{-1}(\R^{n})$ norm of $q$, we find 
 $$\begin{array}{lll}
 \|q\|^{2}_{H^{-1}(\R^{n})}&=& \displaystyle\int_{|\xi|\leq R}|\widehat{q}(\xi)|^{2}<\xi>^{-2}\,d\xi+\displaystyle\int_{|\xi|>R}|\widehat{q}(\xi)|^{2}<\xi>^{-2}\,d\xi\\
 &\leq& C\Big(R^{n}\|\widehat{q}\|_{L^{\infty}(B(0,R))}^{2}+\displaystyle\frac{1}{R^{2}}\|q\|^{2}_{L^{2}(\Omega)} \Big).\\
 \end{array}$$
 Thus, in light of (\ref{equation n 2.45}), we get
 $$  \|q\|^{2}_{H^{-1}(\R^{n})}  \leq C\Big( R^{n}\lambda^{2\delta}\|N_{A_{2},q_{2}}-N_{A_{1},q_{1}}\|^{2}+\displaystyle\frac{R^{n}}{\lambda^{2\beta}}+\displaystyle\frac{1}{R^{2}}   \Big).
 $$
 We choose $R$ such that $R^{n+2}=\lambda^{2\beta}$ and we obtain 
 $$\|q\|_{H^{-1}(\R^{n})}\leq C\Big(R^{k}\|N_{A_{2},q_{2}}-N_{A_{1},q_{1}}\|+\frac{1}{R}\Big),$$
 for some positive constant $k>0$. All the above mentionned statements are valid for $\lambda$ sufficiently large. Assume that there exists $c>0$ such that  $\|N_{A_{2},q_{2}}-N_{A_{1},q_{1}}\|\leq c$. We select 
 $$R=\|N_{A_{2},q_{2}}-N_{A_{1},q_{1}}\|^{-1/(k+1)}.$$
 Thus, $\lambda$ is sufficietly large and we get 
 \begin{equation}\label{equation n 2.46}
 \|q\|_{H^{-1}(\R^{n})}\leq C\|N_{A_{2},q_{2}}-N_{A_{1},q_{1}}\|^{\mu_{2}},\quad \mu_{2}=1/(k+1)\in(0,1).
 \end{equation} 
 This completes the proof of Theorem \ref{prop}.

  
\section{Proof of Theorem \ref{Theorem1.1}}\label{Sec3}

At this stage we are well prepared to deal with the inverse problem under investigation,  that is the identification of $V$ appearing in (\ref{equation n 1.1}) from the knwoledge of $\Lambda_{V}$. Based on Lemma \ref{Lemma 2.1} and Theorem \ref{prop} we prove the main result of this paper. Let us start by stating  the main tool allowing us to prove the stability.

A crucial part of the proof of Theorem \ref{Theorem1.1} is an elliptic Carleman estimate  designed for the elliptic operator  $\Delta$ and given in \cite{[BIB7],[BIB10]} . 
For formulating our Carleman estimate, we shall first set some notaions: let a subboundary $\Gamma_{0}\subset \Gamma$. Assume that there exists a function $\psi\in \mathcal{C}^{2}(\Omega,\R^{n})$  such that
$$\psi(x)>0,\,\,\,x\in\Omega,\qquad \quad |\nabla\psi(x)|>0\,\,\, x\in\overline{\Omega},\qquad    \mbox{and} \quad \p_{\nu}\psi(x)\leq 0\,\,\,x\in\Gamma\setminus  \Gamma_{0}.$$ 
On the other hand, for any given parameter $\beta>0$,  we define the weight function $\eta$ as follows

$$\eta(x)=e^{\beta \,\psi(x)}\qquad x\in\Omega.$$
Then the following Carleman estimate holds true: 
\begin{proposition} (see (\cite{[BIB7],[BIB10]}))
  \label{Proposition 3.1}
There exist $\gamma_{0}>0$ and $C>0$ such that for all $\gamma\geq \gamma_{0}$, we have the following estimate: 
$$
\int_{\Omega}(\gamma|\nabla u(x)|^{2}+\gamma^{3}|u(x)|^{2})e^{2\gamma\eta(x)}\,dx\leq \int_{\Omega}|\Delta u(x)|^{2}e^{2\gamma\eta(x)}\,dx+\int_{\Gamma_{0}}\gamma|\p_{\nu}u(x)|^{2}e^{2\gamma\eta(x)}\,d\sigma,
$$ 
for any $u\in H^{2}(\Omega)$ such that $u(x)=0$ on $\Gamma.$
\end{proposition}
Using the above statement, we are now able to stably retrieve the first order coefficient $V$ from the information given by the DN map $\Lambda_{V}$.

  
\subsection{Stability estimate for the velocity field} Armed with Proposition \ref{Proposition 3.1}, we turn now to proving the main result of this paper. Let us consider  two velocity fields  $V_{1},\,V_{2}\in \mathcal{V}(V_{0},M)$. We define $V=V_{1}-V_{2}$. Our goal is to show  that V stably depends on the DN map $\Lambda_{V_{1}}-\Lambda_{V_{2}}$. In view of  (\ref{equation n 2.37}) and (\ref{equation  n 2.4}) we have the existence of a function $\varphi\in \!W^{3,p_{0}}(\Omega)\cap H_{0}^{1}(\Omega)$ such that 
\begin{equation}\label{equation n 3.47}
V=V_{1}-V_{2}=-2iA'+\nabla(2i\psi)=V'+\nabla \varphi.
\end{equation} 
Then $\varphi$ is solution to the following equation
 \begin{equation*}
\left\{
     \begin{array}{ll}
       \Delta \varphi=\Psi:=\mbox{div} V-\mbox{div} V^{'}=\mbox{div} V, & \mbox{in}\,\,\,\Omega, \\
       \\
       \varphi=0, & \mbox{in}\,\,\,\Gamma.
     \end{array}
   \right.
\end{equation*}
Thanks to (\ref{equation  n 2.4})  and (\ref{equation n 3.47}), we have 
$$\Psi=2(q_{2}-q_{1})+\frac{1}{2}(V^{'})(V_{1}+V_{2})+\displaystyle\frac{1}{2}\nabla \varphi\,(V_{1}+V_{2}).$$
By applying Proposition \ref{Proposition 3.1} to the solution $\varphi$ and using the fact that $\|V_{j}\|_{L^{\infty}(\Omega)}\leq M$,  $j=1,\,2$, we find
\begin{eqnarray}\label{equation n 3.48}
&&\displaystyle\int_{\Omega}\gamma|\nabla \varphi(x)|^{2}e^{2\gamma\eta(x)}\,dx \leq \displaystyle\int_{\Omega}|\Delta \varphi(x)|^{2}e^{2\gamma\eta(x)}\,dx+\displaystyle\int_{\Gamma_{0}}\gamma|\p_{\nu}\varphi(x)|^{2}e^{2\gamma\eta(x)}\,d\sigma\cr
\quad\quad &&\quad \leq C\displaystyle\int_{\Omega}(|(q_{2}-q_{1})(x)|^{2}+|V'(x)|^{2}+|\nabla\varphi(x)|^{2})e^{2\gamma\eta(x)}dx+\displaystyle\int_{\Gamma_{0}}\gamma|\p_{\nu}\varphi(x)|^{2}e^{2\gamma \eta(x)}\,d\sigma.
\end{eqnarray}
By taking $\gamma$ sufficiently large, (\ref{equation n 3.48}) immediately yields
\begin{equation*}
\int_{\Omega} \gamma |\nabla \varphi(x)|^{2} e^{2 \gamma\eta(x)}\,dx\leq  C\displaystyle\int_{\Omega}(|(q_{2}-q_{1})(x)|^{2}+|V'(x)|^{2})e^{2\gamma\eta(x)}dx+\displaystyle\int_{\Gamma_{0}}\gamma|\p_{\nu}\varphi(x)|^{2}e^{2\gamma \eta(x)}\,d\sigma.
\end{equation*}
This implies that 
\begin{equation}\label{equation n 3.49}
\|\nabla \varphi\|^{2}_{L^{2}(\Omega)}\leq C \|q_{2}-q_{1}\|^{2}_{L^{2}(\Omega)}+\|V'\|^{2}_{L^{2}(\Omega)}+\|\p_{\nu}\varphi\|_{L^{2}(\Gamma_{0})}^{2}.
\end{equation}
By interpolation and since $\|q_{2}-q_{1}\|_{W^{1,\infty}(\Omega)}\leq M$, it follows from  (\ref{equation n 2.46}) that
\begin{eqnarray}\label{equation n 3.50}
\|q_{2}-q_{1}\|_{L^{2}(\Omega)} \leq\|q_{2}-q_{1}\|^{1/2}_{H^{1}(\Omega)} \|q_{2}-q_{1}\|_{H^{-1}(\Omega)}^{1/2}\leq  C \|N_{A_{2},q_{2}}-N_{A_{1},q_{1}}\|^{\kappa_{1}},
\end{eqnarray}
for some $\kappa_{1}\in (0,1).$ Moreover, from what has already been shown in Section \ref{subsub2}, it is readily seen that 
\begin{equation}\label{equation n 3.51}
\|V'\|_{L^{2}(\Omega)}\leq  C\|A'\|_{L^{2}(\Omega)}\leq C \|N_{A_{1},q_{1}}-N_{A_{2},q_{2}}\|^{\kappa_{2}},
\end{equation}
for some $\kappa_{2}>0$.  On the other hand, owing to the assumption that $V=V_{1}-V_{2}=0$ on $\Gamma$ and taking advantage of Trace's Theorem, one gets 
\begin{equation}\label{equation n 3.52}
\|\p_{\nu}\varphi\|_{L^{2}(\Gamma_{0})}\leq \|V'\|_{L^{2}(\Gamma)}\leq \|V'\|_{H^{1}(\Omega)}\leq \| V'\|_{L^{2}(\Omega)}^{1/2}\|V'\|^{1/2}_{H^{2}(\Omega)}\leq C \|N_{A_{1},q_{1}}-N_{A_{2},q_{2}}\|^{\kappa_{3}},
\end{equation}
for some $\kappa_{3}>0$. In view of (\ref{equation n 3.49})--(\ref{equation n 3.52}), it is easily understood that 
$$\begin{array}{lll}
\|V_{1}-V_{2}\|_{L^{2}(\Omega)}\leq \|V'\|_{L^{2}(\Omega)}+\|\nabla \varphi\|_{L^{2}(\Omega)}
\leq \|N_{A_{2},q_{2}}-N_{A_{1},q_{1}}\|^{\kappa},
\end{array}$$
where $\kappa:=\min\, (\kappa_{1},\kappa_{2},\kappa_{3})$. From (\ref{equation  n 2.5}) we deduce the desired result. 



\begin{thebibliography}{99}
\bibitem{[BA6]} \textsc{M. Bellassoued, } {\it Stable determination of coefficients in the dynamical Schr\"odinger equation in a magnetic field, } Inverse Problems, Volume 33, Number 5, (2017).
\bibitem{[BIB1]} \textsc{M. Bellassoued, I. Ben A\"icha, }{\it Stable determination outside a cloaking region of two time-dependent coefficients in an hyperbolic equation from Dirichlet to Neumann map,}  Mathematical Analysis and Applications, 46-76, Volume 449, Issue 1, (2017).
     
\bibitem{[BIB2]} \textsc{M. Bellassoued, I. Ben A\"icha, }{\it Uniqueness for an inverse problem for a dissipative wave equation with time dependent coefficient, }  
Arima, 65-78 volume 23, (2016).
  
\bibitem{[BIB3]} \textsc{M. Bellassoued,  H. Benjoud, }{\it Stability estimate for an inverse problem for the wave equation in a magnetic field, }  
 Applicable Analysis, 277-292, Volume 87, Issue 3, (2008).
 
 
\bibitem{[BIB4]} \textsc{M. Bellassoued, M. Choulli, M. Yamamoto, }{\it Stability
estimate for an inverse wave equation and a multidimensional
Borg-Levinson theorem, } 
Differential Equations, 465-494, Volume 247, Issue 2, (2009).  
    
\bibitem{[BIB5]} \textsc{M. Bellassoued, D. Dos Santos Ferreira, }{\it Stability
estimates for the anisotripic wave equation from the Dirichlet-to-Neumann  map,}
Inverse Problems and Imaging, 745 - 773, Volume 5, Issue 4, (2011).
    

\bibitem{[BIB6]}\textsc{M. Bellassoued, Y. Kian, E. Soccorsi } {\it An inverse problem for the magnetic Schr\"odinger equation in infinite cylindrical domains},
arXiv:1605.06599, (2016).   
    
\bibitem{[BIB7]} \textsc{M. Bellassoued, M. Yamamoto, }{\it Carleman Estimates for Anisotropic Hyperbolic Systems in Riemannian Manifolds and Applications, }
Lecture Notes in Mathematical Sciences, (2012).
 
\bibitem{[BIB8]} \textsc{I. Ben A\"icha, }{\it Stability estimate for hyperbolic inverse problem with time dependent coefficient, }
Inverse Problems,  Volume 31, Number 12, (2016).
  

\bibitem{[BIB9]}\textsc{H. Ben Joud} {Stability estimate for an inverse problem for the electro-magnetic wave equation and spectral boundary value problem, } 
Inverse problems, Volume 26, Number 8,  (2010).
  
\bibitem{[BIB10]}\textsc{H. Ben Joud} {A stability estimate for an inverse problem for the    Schr\"odinger equation in a magnetic field from partial boundary measurements},
Inverse problems, Volume 25, Number 4, (2009).
  
           
\bibitem{[BIB11]} \textsc{J. Cheng, G. Nakamura, E. Somersalo, }{\it Uniqueness of identifying the convection term, }  Communications of the Korean Mathematical Society, 405-413, Volume 16, Issue 3,  (2001).
  
 \bibitem{[BIB12]} \textsc{R. Cipolatti, Ivo F. Lopez, }{\it Determination of
 coefficients for a dissipative wave equation via boundary measurements,}
 Mathematical  Analysis and Applications, 317-329, Volume 306, Issue 1, (2005).
    
\bibitem{[BIB13]} \textsc{G. Eskin, }{\it A new approach to hyperbolic inverse
problems, } Inverse problems, 815-831, Volume 22, Number 3, (2006). 
   
 
\bibitem{[BIB14]} \textsc{C. Evans, }{\it   Partial diffrential equations}, American Mathematical Society, Volume 19, (1998).  
 
  
\bibitem{[BIB15]} \textsc{V. Isakov, }{\it An inverse hyperblic problem with many
boundary measurements, } 
Communications in Partial Differential Equations, 1183-1195, Volume 16, Issue 6-7 16, (1991).
    
\bibitem{[BIB16]}  \textsc{V. Isakov,   Z. Sun, } {\it Stability estimates for
hyperbolic inverse problems with local boundary data,} 
Inverse problems, 193-206, Volume 8, Number 2, (1992).
      
  
\bibitem{[BIB17]} \textsc{Y. Kian, }{\it Stability in the determination of a time-dependent coefficient for wave equations from partial data}
Mathematical Analysis and Applications, 408-428, Volume 436, Issue 1, 1, (2016).

\bibitem{[BA5]}\textsc{Y. K Kurylev, M. Lassas}, {\it The multidimensional Gel'fand inverse problem for non-self-adjoint operators, } Inverse problems, 1495-1501, Volume 13, Number 6, (1997).
\bibitem {[BA4]} \textsc{S. Liu, L.Oksanen, } {\it A Lipschitz stable reconstruction formula for the inverse problem for the wave equation,} Transactions of the American Mathematical Society, 319-335,Volume 368, (2016).

\bibitem{[BIB18]} \textsc{V. Pohjola, }{\it A uniqueness result for an inverse problem of the steady state convection diffusion equation, } SIAM Journal on Mathematical Analysis, 2084-2103, Volume 47,  Issue 3, (2015).
 
\bibitem{[BIB19]} \textsc{Rakesh, W. Symes,} {\it Uniqueness for an inverse
problem for the wave equation,} 
Communications in Partial Differential Equations, 87-96, Volume 13, Issue 1,  (1988).

\bibitem{[BA3]} \textsc{A.G. Ramm, Sj\"ostrand}, {\it An inverse inverse problem of the wave equation,}  Mathematische Zeitschrift, 119-130, Volume 206, Issue 1, (1991).

\bibitem{[BA2]} \textsc{R. Salazar, } {\it Determination of time-dependent coefficients for a hyperbolic inverse problem, } Inverse Problems, (2013). 
 \bibitem{[BIB20]} \textsc{M. Salo, }{\it Inverse problems for nonsmooth first oredre perturbation of the Laplacian, } Annales Academiae Scientiarum Fennicae. Mathematica Dissertationes, 139, (2004).    
    
\bibitem{[BIB21]} \textsc{P. Stefanov, G. Uhlmann, }{\it Stability estimates for the
 hyperbolic Dirichlet-to-Neumann map in anisotropic media, } Functional Analysis, 330-358, Volume 154, Issue 2, (1998).
 
 
\bibitem{[BIB22]} \textsc{Z. Sun, }{\it On continuous dependence for an inverse initial boundary value problem for the wave equation,}  Mathematical  Analysis and Applications, 188-204, Volume 150, Issue 1, (1990).

\bibitem{[BIB23]} \textsc{Z. Sun ,}{\it An inverse boundary value problem for Schr\"{o}dinger Operators with Vector Potentials,} Transactions of the American Mathematical Society, 953-969, Volume 1, (1993).   
       
 \bibitem{[BIB24]} \textsc{L. Tzou, }{\it Stability Estimates for Coefficients of Magnetic Schr\"odinger Equation From Full and Partial Boundary Measurements}, Communications in Partial Differential Equations, 1911 - 1952, Volume 33, Issue 11, (2008). 

\bibitem{[BA1]} \textsc{A. Waters, } {\it Stable determination of X-ray transforms of time dependent potentials from the dynamical Dirichlet-to-Neumann map}, Communications in Partial Differential Equations, 39, 12, (2014).




  



\end{thebibliography}
\end{document}